\documentclass[12pt]{amsart}

\usepackage{graphicx}
\usepackage{pst-all}
\usepackage{amsmath}
\usepackage{amsfonts}
\usepackage{amssymb}
\usepackage{amsthm}
\usepackage{url}
\usepackage{subfigure}
\input xy
\xyoption{all}
\usepackage{nicefrac}

\theoremstyle{theorem}
\newtheorem{theorem}{Theorem}[section]
\newtheorem*{theorem*}{Theorem}
\newtheorem{corollary}[theorem]{Corollary}
\newtheorem{lemma}[theorem]{Lemma}
\newtheorem{proposition}[theorem]{Proposition}
\theoremstyle{definition}
\newtheorem{definition}[theorem]{Definition}
\newtheorem{example}[theorem]{Example}
\newtheorem{remark}[theorem]{Remark}

\thanks{The authors are partially supported by the Spanish Ministerio de Econom{\'{i}}a y Competitividad (grant MTM2015-63612-P) and Ministerio de Ciencia, Innovaci{\'{o}}n y Universidades (grant PGC2018-098321-B-I00).}

\begin{document}

\author{H. Barge}
\address{E.T.S. Ingenieros inform\'{a}ticos. Universidad Polit\'{e}cnica de Madrid. 28660 Madrid (Espa{\~{n}}a)}
\email{h.barge@upm.es}

\author{J. J. S\'anchez-Gabites}
\address{Facultad de Ciencias Econ{\'{o}}micas y Empresariales. Universidad Aut{\'{o}}noma de Madrid, Campus Universitario de Cantoblanco. 28049 Madrid (Espa{\~{n}}a)}
\email{JaimeJ.Sanchez@uam.es}
\title{The geometric index and attractors of homeomorphisms of $\mathbb{R}^3$}
\begin{abstract} 
In this paper we focus on compacta $K \subseteq \mathbb{R}^3$ which possess a neighbourhood basis that consists of nested solid tori $T_i$. We call these sets toroidal. In \cite{hecyo1} we defined the genus of a toroidal set as a generalization of the classical notion of genus from knot theory. Here we introduce the self-geometric index of a toroidal set $K$, which captures how each torus $T_{i+1}$ winds inside the previous $T_i$. We use this index in conjunction with the genus to approach the problem of whether a toroidal set can be realized as an attractor for a flow or a homeomorphism of $\mathbb{R}^3$. We obtain a complete characterization of those that can be realized as attractors for flows and exhibit uncountable families of toroidal sets that cannot be realized as attractors for homeomorphisms. 
\end{abstract}

\maketitle

\section*{Introduction}

It is well known that attractors of dynamical systems can exhibit a very complicated topological structure. This phenomenon occurs in dimensions three and higher and for both continuous and discrete dynamical systems. In this paper we will mainly focus on discrete dynamical systems on $\mathbb{R}^3$ generated by homeomorphisms $f : \mathbb{R}^3 \longrightarrow \mathbb{R}^3$, although we will also discuss the case of flows.

To understand how complicated attractors can be we may consider the following ``realizability problem'': find criteria that, given a compactum $K \subseteq \mathbb{R}^3$, decide whether there exists a dynamical system on $\mathbb{R}^3$ for which $K$ is an attractor. Several authors have obtained results about realizability problems of this sort, and as an illustration we may refer the reader to \cite{bhatiaszego1}, \cite{garay1}, \cite{gunthersegal1}, \cite{jimenezllibre1}, \cite{peraltajimenez1}, \cite{mio5}, \cite{sanjurjo1} for the continuous case and to \cite{crovisierrams1}, \cite{duvallhusch1}, \cite{gunther1}, \cite{kato1}, \cite{ortegayo1}, \cite{mio6}, \cite{mio7} for the discrete case. An important idea that emerges from these works is that it is not so much the topology of $K$ itself what plays a crucial role in this problem, but rather how $K$ sits in $\mathbb{R}^3$. This may bring to mind notions such as tameness, cellularity, or other similar ones. 

The realizability problem has an easy answer for flows in $\mathbb{R}^3$ (or any $3$--manifold). Recall that a compact set $K \subseteq \mathbb{R}^3$ is called tame if there exists a homeomorphism of $\mathbb{R}^3$ that sends $K$ onto a polyhedron $P$ and wild otherwise. A set $K$ is weakly tame if its complement is indistinguishable from the complement of a tame set, even though $K$ itself may be wild. More formally, a compact set $K \subseteq \mathbb{R}^3$ is weakly tame if there exists a compact polyhedron $P \subseteq \mathbb{R}^3$ such that $\mathbb{R}^3 - K$ and $\mathbb{R}^3 - P$ are homeomorphic. The following general result holds: a compact set $K \subseteq \mathbb{R}^3$ can be realized as an attractor for a flow if and only if it is weakly tame \cite[Theorem 11, p. 6169]{mio5}.

Another classical notion related to how a compactum sits in the ambient space is cellularity. A compactum $K \subseteq \mathbb{R}^3$ is cellular if it has a neighbourhood basis comprised of $3$--cells (that is, sets homeomorphic to closed $3$--balls). The realizability problem for cellular sets has a very simple answer: every cellular set can be realized as an attractor for a flow and therefore also for a homeomorphism. This is a direct consequence of results of Garay \cite{garay1}. (In fact, a stronger result holds true: a compactum $K \subseteq \mathbb{R}^n$ can be realized as a global attractor for a flow if and only if it is cellular).

Given that cells are handlebodies of genus zero, a natural next step is to consider compacta $K \subseteq \mathbb{R}^3$ that have neighbourhood bases comprised of handlebodies of genus one, namely solid tori. We call these compacta \emph{toroidal}. Some well known attractors, such as solenoids, are toroidal sets. The ultimate motivation for this paper is to analyze the realizability problem for toroidal sets, extending the work initiated in \cite{hecyo1}.

Up to homeomorphism there is only one way to place a (polyhedral) cell in $\mathbb{R}^3$, and in fact one cell inside another. This suggests that, in some sense, all cellular sets sit in $\mathbb{R}^3$ in a similar way. By contrast, knot theory shows that there are many ways to place a torus in $\mathbb{R}^3$ and also to place a torus inside another, where in addition to knotting some winding of the smaller torus inside the bigger one can occur. Heuristically this suggests that there may be many different ways in which toroidal sets sit in $\mathbb{R}^3$, and that it should be possible to describe them to some extent in terms of notions somewhat akin to knottedness and winding. In turn, since the realizability of a compactum as an attractor is related to how it sits in $\mathbb{R}^3$, it is to be expected that there should exist some relation between the knottedness and winding of a toroidal set and its realizability as an attractor.

Guided by the general ideas just described, in \cite{hecyo1} we introduced the genus $g(K)$ of a toroidal set $K$ as a generalization of the classical genus of a knot. Under appropriate circumstances this genus can be computed as $g(K) = \lim_j\ g(T_j)$, where $\{T_j\}$ is any neighbourhood basis of solid tori for $K$ and $g(T_j)$ denotes the genus of the knot represented by the torus $T_j$. Toroidal sets may have infinite genus: this is the case, for instance, of connected sums of infinitely many nontrivial knots or of any nontrivial knotted solenoid. We showed that if a toroidal set can be realized as an attractor for a homeomorphism then it must have finite genus, from which it follows that the toroidal sets just mentioned cannot be realized as such. Thus, the genus of a toroidal set provides an obstruction to the realizability problem. It does not provide a complete characterization, however, since there exist toroidal sets with finite genus that cannot be realized as attractors, such as some generalized solenoids studied by G\"unther \cite{gunther1} and many more examples that we shall construct below.

In this paper we investigate yet another feature of toroidal sets which is related to the winding of a torus inside another. Measuring the winding algebraically (taking into account the sense in which the inner torus winds) is not particularly interesting because it essentially reproduces the \v{C}ech cohomology of the toroidal set. Therefore we turn to another measure of the winding which was first introduced by Schubert \cite[\S 9]{schubert} under the name ``order''. Consider a simple closed curve $\gamma$ contained in the interior of a solid torus $T$. The geometric index of $\gamma$ with respect to $T$ is defined as the minimum number of points of intersection of $\gamma$ with any meridional disk $D$ of $T$. Now, if $T_2$ is a solid torus contained in the interior of another solid torus $T_1$, the \emph{geometric index} of $T_2$ with respect to $T_1$ is defined as the geometric index of any core curve of $T_2$ with respect to $T_1$. We denote this index by $N(T_2,T_1)$. The geometric index has a multiplicative property that will play a crucial role in this paper: if $T_3 \subseteq T_2 \subseteq T_1$ are three nested solid tori, then $N(T_3,T_1) = N(T_3,T_2) \cdot N(T_2,T_1)$.  

To obtain an analogue of the geometric index in the realm of toroidal sets suppose that $K$ is a toroidal set written as the intersection of a nested sequence $\{T_j\}$ of solid tori contained in a solid torus $T$. By analogy with the genus of a toroidal set mentioned earlier, a reasonable attempt to define the geometric index of $K$ inside $T$ would be to consider some sort of limit of the geometric indices $N(T_j,T)$. Typically the sequence $N(T_j,T)$ may grow exponentially (as in the dyadic solenoid, for example), suggesting some formula of the sort $\lim_j\ \ln\ N(T_j,T)/j$. However, this runs into difficulties since the limit depends on the sequence $\{T_j\}$, and so a subtler approach is needed. Consider a prime number $p$ and say that ``$p$ divides the geometric index of $K$'' if the exponent of $p$ in the prime factor decomposition of the integer $N(T_j,T)$ goes to infinity as $j$ goes to infinity. This captures some form of weak exponential behaviour of the sequence $N(T_j,T)$. It is easy to show that the multiplicative property of the geometric index ensures that this property is independent of the particular basis $\{T_j\}$. In fact it is also independent of the torus $T$, so one is led to abandon $T$ altogether and define the geometric index of $K$ ``in itself''. Summing up, it makes heuristic and mathematical sense to speak of the ``prime divisors'', or prime factors, of the self-geometric index of $K$.

As it is often the case, for technical reasons the actual definition of the self-geometric index that we present in Section \ref{sec:defn} is quite more abstract and seemingly removed from these heuristic considerations. After discussing some simple algebra in Section \ref{sec:algebra} we define the self-geometric index of a toroidal set $K$ as a certain subgroup $\mathcal{N}(K)$ of the (additive) group of the real numbers. The prime divisors of the index introduced earlier correspond to those prime numbers $p$ such that every element in $\mathcal{N}(K)$ is divisible by $p$. In general the whole group $\mathcal{N}(K)$ contains more information than its prime factors, but in certain cases $\mathcal{N}(K)$ behaves essentially as an integer number in that it has finitely many prime factors and can be completely recovered from them. We will then say that $\mathcal{N}(K)$ is ``number-like'' and, in the particular case when $\mathcal{N}(K)$ has no prime factors, write $\mathcal{N}(K) \sim 1$.

Returning back to dynamics and the realizability problem, we will show that if a toroidal set is an attractor for a homeomorphism then its self-geometric index is number-like. Using these properties we will exhibit examples (in fact, uncountably many different examples) of toroidal sets that cannot be realized as attractors for homeomorphisms. These examples all have genus zero, and so it is not possible to establish their non-attracting nature using our earlier techniques from \cite{hecyo1}. We will also characterize those toroidal sets that can be realized as attractors for flows as those that have finite genus and whose self-index is $\sim 1$. Finally, we will obtain partial converses to the result mentioned earlier that a toroidal set that can be realized as an attractor has finite genus.

The paper is organized as follows. In Section \ref{sec:prelim} we recall some definitions and suggest some bibliography. Section \ref{sec:algebra} discusses some simple algebra. The definition of the self-geometric index of a toroidal set is given in Section \ref{sec:defn}, where we also establish some of its basic properties. Section \ref{sec:wtame} discusses weakly tame toroidal sets and characterizes them in terms of the genus and the self-index. Our interest in the weakly tame case stems from the fact mentioned earlier that this is precisely the class of sets that can be realized as attractors for flows in $\mathbb{R}^3$. Section \ref{sec:construct} shows how to construct families of toroidal sets having a prescribed cohomology and self-index. Finally, Section \ref{sec:dyn} returns to dynamics and establishes the results mentioned in the previous paragraph.

\section{Background definitions} \label{sec:prelim}

A \emph{toroidal} set is a compact set $K \subseteq \mathbb{R}^3$ that is not cellular and has a neighbourhood basis comprised of solid tori. In this paper we will always assume that sets other than toroidal sets are polyhedral (toroidal sets, in general, can have a very complicated structure). For maps we adopt the opposite convention: they will never be assumed to be piecewise linear unless explicitly stated otherwise. A taming result of Moise ensures that a toroidal set has a neighbourhood basis of polyhedral solid tori (see \cite[p. 5]{hecyo1}). We say that two solid tori $T_1$ and $T_2$ are \emph{nested} if $T_2$ is contained in the interior of $T_1$. Here the word ``interior'' may be taken in the topological sense or in the manifold sense: both agree by the invariance of domain theorem. A \emph{standard basis} for a toroidal set $K$ is a sequence of (polyhedral) nested solid tori $\{T_j\}$ whose intersection is $K$.

A framing of a solid torus $T \subseteq \mathbb{R}^3$ is a particular piecewise linear homeomorphism $f : \mathbb{D}^2 \times \mathbb{S}^1 \longrightarrow T$. A \emph{core} curve of $T$ is a curve $\gamma$ of the form $f(0 \times \mathbb{S}^1)$ for some framing $f$. Notice that $\gamma$ lies in the interior of $T$ and $T$ is a regular neighbourhood of $\gamma$ (recall our convention about the polyhedral nature of sets unless otherwise stated explicitly). Core curves are not unique, but any two core curves are related by an isotopy of $T$ that is the identity on $\partial T$. A meridian of $T$ is a curve of the form $f(\partial \mathbb{D}^2 \times *)$ for some framing of $T$ (here $*$ denotes any point in $\mathbb{S}^1$). Alternatively, meridians can be characterized as those simple closed curves in $\partial T$ which bound a disk in $T$ but not in $\partial T$. A \emph{meridional disk} of $T$ is a disk $D \subseteq T$ that is properly embedded (that is, $D \cap \partial T = \partial D$) and whose boundary $\partial D$ is a meridian of $T$. 

Consider a pair of nested solid tori $T_2 \subseteq T_1$. The \emph{geometric index of $T_2$ in $T_1$} is the minimum number of points of intersection of a core curve $\gamma$ of $T_2$ with any meridional disk $D$ of $T_1$. The minimum is taken over all meridional disks that intersect $\gamma$ transversely, so that $D \cap \gamma$ consists of finitely many points. Since any two core curves of $T_2$ are related by an isotopy of $T_2$ that is the identity on $\partial T_2$ and can therefore be extended to an isotopy of $T_1$, it follows that this definition is independent of the particular core curve $\gamma$. The geometric index is multiplicative: is $T_3 \subseteq T_2 \subseteq T_1$ are three nested solid tori, then $N(T_3,T_1) = N(T_3,T_2) \cdot N(T_2,T_1)$ (see \cite[Satz 3, p. 175]{schubert}). We will make extensive use of this property.

Suppose $T_2 \subseteq T_1$ and $T'_2 \subseteq T'_1$ are two pairs of nested solid tori. It is clear from the definition of the geometric index that if there exists a piecewise linear homeomorphism between the pairs $(T_1,T_2)$ and $(T'_1,T'_2)$ then the geometric indices $N(T_2,T_1)$ and $N(T'_2,T'_1)$ are equal. We will need a slightly stronger statement, however, which is given in Lemma \ref{lem:rel_hom} below. We have opted for a quick, if not particularly neat, proof:

\begin{lemma} \label{lem:rel_hom} Suppose that $f : (T_1,T_2) \longrightarrow (T'_1,T'_2)$ is a homeomorphism between two pairs of nested polyhedral tori. Assume that $f$ is piecewise linear on $T_1 - {\rm int}\ T_2$. Then the geometric indices of the two pairs are equal.
\end{lemma}
\begin{proof} By an approximation theorem of Bing (see for instance \cite[Theorem 4, p. 149]{bing2}) there exists a piecewise linear homeomorphism $f' : T_1 \longrightarrow T'_1$ which coincides with $f$ on $T_1 - {\rm int}\ T_2$. In particular it sends $T_2$ onto $T'_2$ and therefore provides a piecewise linear homeomorphism from $(T_1,T_2)$ onto $(T'_1,T'_2)$. The result follows.
\end{proof}

A natural way of constructing toroidal sets is, of course, by taking a nested sequence of tori $\{T_j\}$ and letting $K := \bigcap_{j \geq 0} T_j$. Although such a $K$ certainly has a neighbourhood basis of solid tori, to guarantee that $K$ is a toroidal set we must also make sure that it is not cellular. The following proposition provides an easy criterion to check this in terms of the geometric index.

\begin{proposition} \label{prop:cellular} Let $K$ be defined as the intersection of a nested sequence of solid tori $\{T_j\}$. Then $K$ is cellular if and only if $N(T_{j+1},T_j) = 0$ for infinitely many $j$.
\end{proposition}

For the proof we make use of the following fact: if $T_2 \subseteq T_1$ is a nested pair of solid tori, then $N(T_2,T_1) = 0$ if and only if $T_2$ is contained in a $3$--cell $B$ inside $T$. This should be intuitively clear, and a very sketchy argument will be given in the proof of the proposition. A formal proof can be found in \cite[Satz 1, p. 171]{schubert}.

\begin{proof}[Proof of Proposition \ref{prop:cellular}] Assume first that $K$ is cellular. Choose any $T_{j_0}$. Since $K$ is cellular, it has a neighbourhood $B$ which is a $3$--cell contained in ${\rm int}\ T_{j_0}$. In turn, take $j_1$ such that $T_{j_1} \subseteq {\rm int}\ B$. Shrink the cell $B$ by an ambient isotopy of $T_{j_0}$ (relative to $\partial T_{j_0}$) until it becomes extremely small. It is then clear that there is a meridional disk $D$ of $T_{j_0}$ that is disjoint from the shrinked cell. Carrying $D$ along with the reverse isotopy one obtains a meridional disk of $T_{j_0}$ that is disjoint from the original cell $B$, and consequently also from $T_{j_1}$. Thus the geometric index of $T_{j_1}$ in $T_{j_0}$ is zero. The multiplicativity of the index then implies that $N(T_{j+1},T_j) = 0$ for some $j_0 \leq j < j_1$.

Conversely, assume that $N(T_{j+1},T_j) = 0$ for infinitely many $j$. Choose any $T_{j_0}$ and pick $j \geq j_0$ such that $N(T_{j+1},T_j) = 0$. Then $N(T_{j+1},T_{j_0}) = 0$ too by the multiplicative property of the geometric index. This means that there exists a meridional disk $D$ of $T_{j_0}$ such that $T_{j+1}$ is disjoint from $D$. Thickening $D$ to a regular neighbourhood (relative to $\partial T_{j_0}$) small enough that it is still disjoint from $T_{j+1}$ and then removing its interior from $T_{j_0}$ produces a polyhedral $3$--cell $B$ contained in $T_{j_0}$. Thus $K$ has a neighbourhood basis of $3$--cells and is therefore cellular.
\end{proof}

{\it Warning.} Whenever we construct a toroidal set as an intersection of a sequence of nested tori we will never check explicitly that it is not cellular, since this will always be a direct consequence of the above proposition.
\medskip

Consider once more a nested pair of solid tori $T_2 \subseteq T_1$. The inclusion induces a map $H^1(T_1;\mathbb{Z}) \longrightarrow H^1(T_2;\mathbb{Z})$ which, after identifying each of the cohomology groups with $\mathbb{Z}$ with an appropriate orientation, simply becomes multiplication by a nonnegative integer $w \geq 0$. This is usually called the \emph{winding number} of $T_2$ inside $T_1$. It is not difficult to see that, if $D$ is any meridional disk of $T_1$ and $\gamma$ is a core curve of $T_2$ transverse to $D$, then $w$ is the (absolute value of the) number of intersections of $\gamma$ with $D$ counted with a sign according to the sense in which $\gamma$ crosses $D$. In particular, choosing $D$ to minimize the number of points of intersection with $\gamma$ one has that $w \leq N(T_2,T_1)$ and also that $w$ and $N(T_2,T_1)$ have the same parity.

The (\v{C}ech) cohomology of toroidal sets will play a role in this paper, so we devote a few lines to describe it now. Let $\{T_j\}$ be a standard basis for a toroidal set $K$. By the continuity property of \v{C}ech cohomology, $\check{H}^q(K;\mathbb{Z})$ is the direct limit of the direct sequence $\check{H}^q(T_j;\mathbb{Z}) \longrightarrow \check{H}^q(T_{j+1};\mathbb{Z})$, where the arrows are induced by the inclusions $T_{j+1} \subseteq T_j$. In degrees $q \geq 2$ clearly $\check{H}^q(K;\mathbb{Z}) = 0$. In degree $q = 1$, and according to the discussion in the previous paragraph, each of these arrows can be identified with $\mathbb{Z} \stackrel{\cdot w_j}{\longrightarrow} \mathbb{Z}$, where $w_j$ is the winding number of $T_{j+1}$ inside $T_j$. It is very easy to check (or see a proof in \cite[Proposition 1.6]{hecyo1}) that there are three mutually exclusive possibilities for $\check{H}^1(K;\mathbb{Z})$:
\begin{itemize}
	\item[(i)] If infinitely many of the $w_j$ vanish, then $\check{H}^1(K;\mathbb{Z}) = 0$. We then say that $K$ is \emph{trivial}.
	\item[(ii)] If $w_j = 1$ from some $j$ onwards, then $\check{H}^1(K;\mathbb{Z}) = \mathbb{Z}$.
	\item[(iii)] Otherwise $\check{H}^1(K;\mathbb{Z})$ is not finitely generated.
\end{itemize}

We conclude this section by recalling the definition of the genus of a toroidal set. For information about knot theory we refer the reader to the books by Burde and Zieschang \cite{burdezieschang1}, Lickorish \cite{Lickorish1} or Rolfsen \cite{rolfsen1}. Define the genus $g(T)$ of a solid torus $T$ as the genus of a core curve of $T$ (this definition is correct because core curves are unique up to isotopy). For a toroidal set $K \subseteq \mathbb{R}^3$ define its \emph{genus} $g(K)$ as the minimum $g = 0,1,\ldots, \infty$ such that $K$ has arbitrarily small neighbourhoods that are solid tori of genus $\leq g$. Notice that if $K$ has finite genus $g$ then it has a standard basis $\{T_j\}$ such that $g(T_j) = g$ for every $j$. Clearly a toroidal set $K$ has genus zero if and only if it has a neighbourhood basis of solid tori $\{T_j\}$ of genus zero, which means that they are unknotted. These toroidal sets we shall call \emph{unknotted}.

\section{Some algebraic preliminaries} \label{sec:algebra}

We shall begin with some very simple algebra. Given a sequence of nonnegative integers $\{m_j\}_{j \geq 1}$, consider the direct sequence  \begin{equation} \label{eq:dir_lim_Z} \xymatrix{\mathbb{Z} \ar[r]^{\cdot m_1} & \mathbb{Z} \ar[r]^{\cdot m_2} \ar[r] & \mathbb{Z} \ar[r] & \ldots },\end{equation} where each arrow is multiplication by $m_j$. Denote by $G$ the direct limit of this sequence. The following mutually exclusive alternatives hold:
\begin{itemize}
	\item[(i)] $G = 0$ if and only if $m_j = 0$ for infinitely many $j$.
	\item[(ii)] $G = \mathbb{Z}$ if and only if $m_j = 1$ from some $j$ onwards.
	\item[(iii)] Otherwise $G$ is not finitely generated.
\end{itemize}

Let $p$ be a prime number. If every element in $G$ is divisible by $p$ we then say that \emph{$p$ divides $G$}, or that \emph{$p$ is a prime factor of $G$}, and write $p | G$. The following proposition provides an alternative characterization of this notion. The proof is extremely simple but we include it for completeness:

\begin{proposition} \label{prop:divisors} With the above notation, $p | G$ if and only if $p$ divides infinitely many of the $m_j$.
\end{proposition}
\begin{proof} It will be convenient to attach a subscript $(j)$ to each copy of $\mathbb{Z}$ in \eqref{eq:dir_lim_Z} to distinguish them. We choose the notation so that the arrows read $\xymatrix{\mathbb{Z}_{(j)} \ar[r]^{\cdot m_j} & \mathbb{Z}_{(j+1)}}$. Any element $z \in G$ is represented by some $z_{(j_0)} \in \mathbb{Z}_{(j_0)}$. Of course, $z$ is also represented by any of its succesive images $z_{(j)} \in \mathbb{Z}_{(j)}$ along the sequence. The relation between $z_{(j_0)}$ and $z_{(j)}$ is just \begin{equation} \label{eq:jj0} z_{(j)} = z_{(j_0)} \cdot m_{j_0} \cdot m_{j_0+1} \cdot \ldots \cdot m_{j-1}.\end{equation}

Suppose first that $p$ divides infinitely many of the $m_j$. Let $z$ be represented by $z_{(j_0)} \in \mathbb{Z}_{(j_0)}$. Our assumption implies that there exists $j \geq j_0$ such that $p | m_j$. From \eqref{eq:jj0} we see that $p | z_{(j+1)}$, and then $z'_{(j+1)} := z_{(j+1)}/p \in \mathbb{Z}_{(j+1)}$ represents an element $z' \in G$ such that $pz' = z$.

Conversely, assume every element in $G$ is divisible by $p$. Fix $j_0$ and consider the element $z \in G$ represented by $z_{(j_0)} = 1 \in \mathbb{Z}_{(j_0)}$. By assumption there exists $z' \in G$ such that $pz' = z$. Then for big enough $j$ we must have $pz'_{(j)} = z_{(j)}$, so in particular $p | z_{(j)}$. But then from \eqref{eq:jj0} we see that $p$ divides some $m_k$ with $j_0 \leq k < j$. Since $j_0$ was arbitrary, it follows that $p$ divides infinitely many of the $m_j$.
\end{proof}

Using the proposition one can easily see that $G$ may not have any prime factors at all or, at the other extreme, it may have infinitely many of them. In fact, it is easy to produce groups $G$ having any prescribed set $P$ of prime numbers as prime divisors:

\begin{example} \label{ex:p_div} Given any set $P$ of prime numbers, there exists a group $G$ whose prime divisors are precisely the elements of $P$. For instance:
\begin{itemize}
	\item[(i)] When $P = \emptyset$ choose $m_j$ to be the $j$th prime number.
	\item[(ii)] When $P$ is finite choose all the $m_j$ to be equal to $m$, the product of all the elements in $P$.
	\item[(iii)] When $P$ is infinite, let $\{p_k\}$ be an enumeration of its elements and choose $m_j := p_1 \cdot \ldots \cdot p_j$.
\end{itemize}
\end{example}

If all the $m_j$ are equal to some number $m \geq 1$ for big enough $j$ we will say that $G$ is \emph{number-like}\footnote{This terminology is slightly abusive, because as introduced it describes not an intrinsic property of $G$ but rather of the particular description of $G$ afforded by the direct sequence \eqref{eq:dir_lim_Z} of $G$.}. It is well known that $G$ can then be identified with the $m$--adic rationals; that is, the subgroup of $\mathbb{Q}$ that consists of numbers of the form $\nicefrac{z}{m^k}$ with $z \in \mathbb{Z}$ and $k$ a nonnegative integer. This subgroup is usually denoted by $\mathbb{Z}[\nicefrac{1}{m}]$. We are now going to show that the prime divisors of a number-like $G$ determine it completely:

\begin{proposition} \label{prop:n_like} Let $G$ be number-like. Then it has finitely many prime divisors $p_1, \ldots, p_r$ and $G = \mathbb{Z}[\nicefrac{1}{(p_1 \cdot \ldots \cdot p_r)}]$.
\end{proposition}
\begin{proof} Since $G$ is number-like, it is the direct limit of \eqref{eq:dir_lim_Z} with $m_j = m \geq 1$ for large enough $j$. Then by Proposition \ref{prop:divisors} the prime divisors of $G$ are exactly the prime divisors of $m$ (in the sense of ordinary arithmetics). In particular $G$ has finitely many prime divisors (and it has none exactly when $m = 1$, or equivalently when $G = \mathbb{Z}$). Write $m = p_1^{n_1} \cdot \ldots \cdot {p_r}^{n_r}$ for appropriate exponents $n_i \geq 1$. Also, set $m_0 := p_1 \cdot \ldots \cdot p_r$ for brevity. Notice that $m_0 | m$.

Now we prove that $G = \mathbb{Z}[\nicefrac{1}{m_0}]$. Since $G$ can be identified with the subgroup of $m$--adic rationals, we only need to show that the $m$--adic rationals are the same as the $m_0$--adic rationals. First, every $m_0$--adic rational can be written as an $m$--adic rational as follows: \[\frac{z}{m_0^k} = \frac{z}{m^k} \cdot \frac{m^k}{m_0^k} = \frac{z \cdot \left(\nicefrac{m}{m_0}\right)^k}{m^k}  \in \mathbb{Z}[\nicefrac{1}{m}]\] because $m_0 | m$. To prove that every $m$--adic rational is an $m_0$--adic rational we first perform the following computation. Letting $n = \max\ n_i$, we can write $m$ as follows: \[m = p_1^{n_1} \cdot \ldots \cdot {p_r}^{n_r} =  \frac{(p_1 \cdot \ldots \cdot p_r)^n}{p_1^{n-n_1} \cdot \ldots \cdot p_r^{n-n_r}} = \frac{m_0^n}{m'}\] where $m' = p_1^{n-n_1} \cdot \ldots \cdot p_r^{n-n_r} \in \mathbb{Z}$ because of the definition of $n$. Then any $m$--adic rational can be expressed \[\frac{z}{m^k} = \frac{z}{(\nicefrac{m_0^n}{m'})^k} = \frac{z \cdot m'^k}{m_0^{nk}} \in \mathbb{Z}[\nicefrac{1}{m_0}].\]
\end{proof}

We thus see that when $G$ is number-like it conveys exactly the same information as the number $p_1 \cdot \ldots \cdot p_r$. Hopefully this justifies our terminology. We shall sometimes write $G \sim p_1 \cdot \ldots \cdot p_r$. As mentioned in the proof of the proposition, a somewhat degenerate case is $G = \mathbb{Z}$, which is indeed number-like but has no prime factors. We shall then write $G \sim 1$. Notice that $\mathbb{Z}$ is the only number-like $G$ that does not have any prime factors.

We conclude this section with a very simple observation. There are two manipulations that can be performed on \eqref{eq:dir_lim_Z} which clearly do not alter its limit $G$. The first one is just removing a finite number of terms from the beginning of the sequence. The second consists in organizing the arrows in \eqref{eq:dir_lim_Z} into groups of $k_1,k_2,\ldots$ consecutive arrows and replacing each group with a single arrow which is the composition of all the arrows in the group, obtaining the direct sequence \[\xymatrix@=3cm{\mathbb{Z} \ar[r]^{\cdot m_1 \cdot \ldots \cdot m_{k_1}} & \mathbb{Z} \ar[r]^{\cdot m_{k_1+1} \cdot \ldots \cdot m_{k_1 + k_2}} & \mathbb{Z} \ar[r] & \ldots }\] These two operations allow to put sequence \eqref{eq:dir_lim_Z} in a somewhat canonical form that is sometimes convenient:

\begin{lemma} \label{lem:canonical} Given any group $G$ represented as the direct limit of \eqref{eq:dir_lim_Z} we may assume that the $\{m_j\}$ are in one of the mutually exclusive forms:
\begin{itemize}
	\item[(i)] All the $m_j = 0$, in which case $G = 0$.
	\item[(ii)] All the $m_j = 1$, in which case $G = \mathbb{Z}$.
	\item[(iii)] All the $m_j \geq 2$, in which case $G$ is not finitely generated.
\end{itemize}
\end{lemma}
\begin{proof} (i) If infinitely many of the $m_j$ are zero, then by grouping the arrows in such a way that each group contains at least one zero arrow we may assume that $m_j = 0$ for all $j$. Clearly then $G = 0$.

(ii) If (i) does not hold, then only finitely many of the $m_j$ are zero. We may discard them and assume directly that $m_j \geq 1$ for every $j$. If only finitely many of the $m_j$ satisfy $m_j \geq 2$ then we may also discard them and obtain $m_j = 1$ for every $j$. Clearly then $G = \mathbb{Z}$.

(iii) If (i) and (ii) do not hold, then infinitely many of the $m_j$ satisfy $m_j \geq 2$ (and $m_j \geq 1$ for every $j$). Then by grouping the arrows in \eqref{eq:dir_lim_Z} in such a way that each group contains at least one element $\geq 2$ we may assume that $m_j \geq 2$ for every $j$. It is then very easy to see that $G$ is not finitely generated (the argument is essentially the same as in \cite[Proposition 1.6]{hecyo1}).
\end{proof}

\section{The self-index of a toroidal set} \label{sec:defn}

Let $K$ be a toroidal set and let $\{T_j\}$ be a standard basis for $K$. Denote by $N_j$ the geometric index of $T_{j+1}$ inside $T_j$ and consider the direct sequence \[ \mathcal{N}\{T_j\} : \xymatrix{\mathbb{Z} \ar[r]^{\cdot N_1} & \mathbb{Z} \ar[r] & \ldots \ar[r] & \mathbb{Z} \ar[r]^{\cdot N_j} & \mathbb{Z} \ar[r] & \dots} \]

\begin{proposition} \label{prop:indep} The direct limit of $\mathcal{N}\{T_j\}$ is independent of the basis $\{T_j\}$ chosen to compute it.
\end{proposition}
\begin{proof} First let us show that replacing $\{T_j\}$ with a subsequence $\{T_{j_k}\}$ leaves the direct limit of $\mathcal{N}\{T_j\}$ unchanged. The multiplicativity of the geometric index ensures that \[N(T_{j_2},T_{j_1}) = N(T_{j_2},T_{j_2-1}) \cdot \ldots \cdot N(T_{j_1+1},T_{j_1}) = N_{j_2-1} \cdot \ldots \cdot N_{j_1}\] and similarly for the following $j_k$. Thus the sequence $\mathcal{N}\{T_{j_k}\}$ associated to $\{T_{j_k}\}$ is precisely \[\mathcal{N}\{T_{j_k}\} : \xymatrix@=3cm{\mathbb{Z} \ar[r]^{\cdot N_{j_1} \cdot \ldots \cdot N_{j_2-1}} & \mathbb{Z} \ar[r]^{\cdot N_{j_2} \cdot \ldots \cdot N_{j_3-1}} & \mathbb{Z} \ar[r] & \ldots }\] This can be obtained from the original $\mathcal{N}\{T_j\}$ by removing the first $j_1-1$ terms and then grouping the remaining arrows into the blocks $N_{j_k} \cdot \ldots \cdot N_{j_{k+1}-1}$. These two manipulations do not alter the direct limit of the original sequence, and so $\mathcal{N}\{T_j\}$ and $\mathcal{N}\{T_{j_k}\}$ have the same direct limit, as claimed above.

Now let $\{T'_j\}$ be another standard basis for $K$. We want to show that the direct limits of $\mathcal{N}\{T_j\}$ and $\mathcal{N}\{T'_j\}$ coincide. After passing to suitable subsequences of $\{T_j\}$ and $\{T'_j\}$ we may assume that $T_j \supseteq T'_j \supseteq T_{j+1}$ for every $j$. By the previous paragraph there is no loss in generality in doing this. As before we use the notation $N_j = N(T_{j+1},T_j)$ and $N'_j = N(T'_{j+1},T'_j)$. Also, set \[M_j = N(T'_j,T_j) \quad \quad \text{and} \quad \quad M'_j = N(T_{j+1},T'_j).\] The multiplicative property of the geometric index ensures that $N_j = M_j M'_j$ and $N'_j = M'_j M_{j+1}$. Consider the direct sequence \[\xymatrix{\mathbb{Z} \ar[r]^{\cdot M_1} & \mathbb{Z} \ar[r]^{\cdot M'_1} & \mathbb{Z} \ar[r] & \ldots \ar[r] & \mathbb{Z} \ar[r]^{\cdot M_j} & \mathbb{Z} \ar[r]^{\cdot M'_j} & \mathbb{Z} \ar[r] & \dots} \] Grouping the arrows in pairs yields the sequence \[\xymatrix{\mathbb{Z} \ar[r]^{\cdot N_1} & \mathbb{Z} \ar[r] & \ldots \ar[r] & \mathbb{Z} \ar[r]^{\cdot N_j} & \mathbb{Z} \ar[r] & \dots}\] which is precisely $\mathcal{N}\{T_j\}$, whereas discarding the first arrow and grouping the remaining ones in pairs yields the sequence \[\xymatrix{\mathbb{Z} \ar[r]^{\cdot N'_1} & \mathbb{Z} \ar[r] & \ldots \ar[r] & \mathbb{Z} \ar[r]^{\cdot N'_j} & \mathbb{Z} \ar[r] & \dots}\] which is precisely $\mathcal{N}\{T'_j\}$. It follows that the direct limit of all three sequences is the same, and so in particular $\mathcal{N}\{T_j\}$ and $\mathcal{N}\{T'_j\}$ have the same direct limit, as was to be shown.
\end{proof}

The proposition justifies the correctness of the following definition:

\begin{definition} \label{def:selfindex} Given a toroidal set $K$, its \emph{self-geometric index} (or \emph{self-index} for brevity) is the direct limit of the direct sequence $\mathcal{N}\{T_j\}$ for any standard basis $\{T_j\}$ of $K$. We denote the self-index by $\mathcal{N}(K)$.
\end{definition}

Since a toroidal set is not cellular by definition, it follows from Proposition \ref{prop:cellular} that $N_j \geq 1$ for sufficiently large $j$. Thus $\mathcal{N}(K) \neq 0$ and $\mathcal{N}(K)$ is either $\mathbb{Z}$ or not finitely generated. The case $\mathcal{N}(K) = \mathbb{Z}$, or $\mathcal{N}(K) \sim 1$ in our notational convention, is interesting enough to single out:

\begin{proposition} \label{prop:N1} Let $K \subseteq \mathbb{R}^3$ be a toroidal set. The following three statements are equivalent:
\begin{itemize}
	\item[(i)] $\mathcal{N}(K) \sim 1$.
	\item[(ij)] $N(T_{j+1},T_j) = 1$ for large enough $j$ for some standard basis $\{T_j\}$ for $K$.
	\item[(iii)] $N(T_{j+1},T_j) = 1$ for large enough $j$ for any standard basis $\{T_j\}$ for $K$.
\end{itemize}

Moreover, if any of these conditions holds then $\check{H}^1(K;\mathbb{Z}) = \mathbb{Z}$.
\end{proposition}
\begin{proof} Notice that $\mathcal{N}(K) = \mathbb{Z}$ occurs precisely when $N_j = N(T_{j+1},T_j) = 1$ for large $j$, so this condition must be independent of the basis $\{T_j\}$ since the same is true of $\mathcal{N}(K)$ by Proposition \ref{prop:indep}. This establishes the equivalence of (i), (ii) and (iii). Denote by $w_j$ the winding number of $T_{j+1}$ inside $T_j$. Recall that $w_j$ and $N_j$ have the same parity and also $w_j \leq N_j$. These two conditions (together with $N_j = 1$) force $w_j = 1$, and so $\check{H}^1(K;\mathbb{Z}) = \mathbb{Z}$.
\end{proof}

Mimicking the definitions of the previous section, we say that a prime $p$ divides $\mathcal{N}(K)$ if every element in $\mathcal{N}(K)$ is divisible by $p$ or, more operationally and as a consequence of Proposition \ref{prop:divisors}, if $p|N(T_{j+1},T_j)$ for infinitely many $j$.

\begin{example} \label{ex:gen_sol} Let $K$ be a generalized solenoid defined as the intersection of nested solid tori $\{T_j\}$ such that each $T_{j+1}$ winds monotonically (that is, without turning back) $n_j$ times inside the previous $T_j$ and the diameters of the $T_j$ tend to zero as $j$ increases. To avoid degenerate cases we require $n_j > 1$ for each $j$. Clearly $N(T_{j+1},T_j) = n_j$ for every $j$, and so $\mathcal{N}(K) \not\sim 1$ and $p | \mathcal{N}(K)$ if and only if $p$ is a prime factor of infinitely many of the $n_j$. Depending on our choice of the $n_j$ we may find several scenarios:

(i) If $K$ is an $n$--adic solenoid ($n_j = n$ for every $j$) then $\mathcal{N}(K)$ is number-like and the prime divisors of $\mathcal{N}(K)$ are precisely the prime divisors of $n$ (in the ordinary sense of elementary arithmetics).

(ii) One may pick the $n_j$ in such a way that each prime $p$ appears as a factor of infinitely many of them (for instance, by taking $n_j = j!$). Then every prime divides $\mathcal{N}(K)$.

(iii) At the other extreme, one may choose the $n_j$ in such a way that no prime $p$ is a factor of infinitely many of them, for instance by letting all the $n_j$ be pairwise prime to each other. Then $\mathcal{N}(K)$ has no prime divisors.
\end{example}

The status of the prime $p = 2$ as a factor of $\mathcal{N}(K)$ is somewhat special. The following holds:

\begin{proposition} \label{prop:2} The prime $p = 2$ is a factor of $\mathcal{N}(K)$ if and only if every element in $\check{H}^1(K;\mathbb{Z})$ is divisible by $2$. In particular, having $p = 2$ as a prime factor of $\mathcal{N}(K)$ is a topological property of toroidal sets.
\end{proposition}
\begin{proof} Let $\{T_j\}$ be a standard basis for $K$ and denote by $w_j$ and $N_j$ the winding number and the geometric index of each $T_{j+1}$ inside $T_j$ respectively. By the continuity property of \v{C}ech cohomology $\check{H}^1(K;\mathbb{Z})$ is the direct limit of the direct sequence \[ \xymatrix{\mathbb{Z} \ar[r]^{\cdot w_1} & \mathbb{Z} \ar[r]^{\cdot w_2} & \ldots \ar[r] & \mathbb{Z} \ar[r]^{\cdot w_j} & \mathbb{Z} \ar[r] & \ldots } \] whereas $\mathcal{N}(K)$ is by definition the direct limit of \[ \xymatrix{\mathbb{Z} \ar[r]^{\cdot N_1} & \mathbb{Z} \ar[r]^{\cdot N_2} & \ldots \ar[r] & \mathbb{Z} \ar[r]^{\cdot N_j} & \mathbb{Z} \ar[r] & \ldots } \] Since the geometric index and the winding number both have the same parity, $p = 2$ divides infinitely many of the $w_j$ if and only if it divides infinitely many of the $N_j$. The result follows from the characterization of Proposition \ref{prop:divisors}.
\end{proof}

It seems reasonable from its definition that the self-index of a toroidal set should be invariant under (local) ambient homeomorphisms. This is indeed true, although there is a slight technical subtlety: since Definition \ref{def:selfindex} involves neighbourhood bases that consist of polyhedral tori (for the geometric index to be defined), it is clear that the self-index is invariant under piecewise linear homeomorphisms, but not necessarily under arbitrary ones. To prove this in general we will need to make use of the deep result that any homeomorphism of a $3$--manifold can be approximated arbitrarily closely by a piecewise linear one.

\begin{proposition} \label{prop:local} Let $K$ and $K'$ be toroidal sets in $\mathbb{R}^3$. Suppose that $f : O \longrightarrow f(O)$ is a homeomorphism defined on an open neighbourhood $O$ of $K$ and $f(K) = K'$. Then $\mathcal{N}(K)$ and $\mathcal{N}(K')$ are equal.
\end{proposition}
\begin{proof} By the theorem on invariance of domain $f$ is an open map and $O' := f(O)$ is a neighbourhood of $K'$. Consider the open $3$--manifolds $O-K$ and $O' - K'$. These are homeomorphic via $f$. Let $\phi : O - K \longrightarrow (0,+\infty)$ be defined by $\phi(p) := {\rm d}(p,K)$, where ${\rm d}$ denotes the usual distance. Clearly $\phi$ is bounded away from zero on every compact subset of $O - K$, and then \cite[Theorem 1, p. 253]{moise2} guarantees that there exists a piecewise linear homeomorphism $g : O - K \longrightarrow O' - K'$ such that ${\rm d}(f(p),g(p)) < \phi(p)$ for every $p \in O - K$. Extend $g$ to all of $O$ by defining $g(p) := f(p)$ for $p \in K$. It is straightforward to check that this extension provides a continuous bijection from $O$ to $O'$, which is therefore a homeomorphism by the theorem on invariance of domain.

Now let $\{T_j\}$ be a standard basis for $K$ with all $T_j$ contained in $O$. Define $T'_j := g(T_j)$. These clearly form a standard basis for $K'$ because $g$ is a homeomorphism. Moreover, by construction $g$ provides homeomorphisms of pairs $g : (T_j,T_{j+1}) \longrightarrow (T'_j,T'_{j+1})$ which are piecewise linear on $T_j - {\rm int}\ T_{j+1}$. Thus by Lemma \ref{lem:rel_hom} the geometric indices $N(T_{j+1},T_j)$ and $N(T'_{j+1},T'_j)$ are equal. The equality $\mathcal{N}(K) = \mathcal{N}(K')$ then follows from Proposition \ref{prop:indep}.
\end{proof}

There is no reason to expect that two homeomorphic (but not ambient homeomorphic) toroidal sets should have the same self-index; that is, that the self-index should be a topological invariant of toroidal sets. However, we do not know how to construct an example that illustrates this. Also, Proposition \ref{prop:2} does show that having $2$ as a prime factor of the self-index is indeed a topological invariant. This prompts the question of whether the self-index enjoys some subtler sort of topological invariance property, perhaps involving only its prime factors or its number-like nature. The case of generalized solenoids is somewhat exceptional in this regard for, as it turns out, generalized solenoids are completely characterized by their self-index:

\begin{proposition} \label{prop:charac_gsol} Let $K$ and $K'$ be two generalized solenoids. Then $K$ is homeomorphic to $K'$ if and only if their self-indices $\mathcal{N}(K)$ and $\mathcal{N}(K')$ are isomorphic (as groups).
\end{proposition}
\begin{proof} Recall that we defined a generalized solenoid as the intersection of a nested sequence of solid tori such that each torus winds monotonically (that is, without doubling back) inside the previous one. The monotonicity condition implies that the geometric index and the winding number of each pair of consecutive tori both coincide. In particular the direct sequences that arise when computing the self-index of a generalized solenoid and its \v{C}ech cohomology in degree one are the same and so their direct limits coincide; that is, $\mathcal{N} = \check{H}^1$. Since generalized solenoids are topologically characterized by their \v{C}ech cohomology (see for instance the argument in \cite[p. 198]{mccord1}), the result follows.
\end{proof}

\section{Weakly tame toroidal sets} \label{sec:wtame}

In this section we focus on weakly tame toroidal sets. Recall that, as mentioned in the Introduction, a compact set $K \subseteq \mathbb{R}^3$ is called \emph{weakly tame} if there exists a compact polyhedron $P \subseteq \mathbb{R}^3$ such that $\mathbb{R}^3 - K$ and $\mathbb{R}^3 - P$ are homeomorphic.

The following criterion allows one to recognize when a toroidal set is weakly tame in terms of any one of its standard neighbourhood bases. We recall that two nested tori $T_2 \subseteq T_1$ are \emph{concentric} if there exists a homeomorphism $T_1 - {\rm int}\ T_2 \cong (\partial T_1) \times [0,1]$ such that $\partial T_1 \ni p \longmapsto (p,0) \in (\partial T_1) \times [0,1]$. 
 
\begin{proposition} \label{prop:charac_wtame} Let $K \subseteq \mathbb{R}^3$ be a toroidal set. Then the following are equivalent:
\begin{itemize}
	\item[(i)] $K$ is weakly tame.
	\item[(ii)] There exists a standard basis $\{T_j\}$ for $K$ such that all the $T_j$ are concentric.
	\item[(iii)] For any standard basis $\{T_j\}$ for $K$, the $T_j$ are concentric for large enough $j$.
\end{itemize}
\end{proposition}
\begin{proof} (i) $\Rightarrow$ (ii) By assumption there exists a homeomorphism $h : \mathbb{R}^3 - P \longrightarrow \mathbb{R}^3 - K$ for some compact polyhedron $P$. We may assume without loss of generality that $h$ is piecewise linear (\cite[Theorem 2, p. 253]{moise2}). Notice that \[H^q(P;\mathbb{Z}) = \tilde{H}_{2-q}(\mathbb{R}^3 - P;\mathbb{Z}) = \tilde{H}_{2-q}(\mathbb{R}^3 - K;\mathbb{Z}) = \check{H}^q(K;\mathbb{Z})\] where the first and last equalities follow from Alexander duality. This shows (setting $q = 0$) that $P$ must be connected and $\check{H}^q(P;\mathbb{Z}) = 0$ for $q \geq 2$. For $q = 1$ we reason as follows. The only possibilities for $\check{H}^1(K;\mathbb{Z})$ are either $0$, $\mathbb{Z}$, or a non-finitely generated group. Since compact polyhedra have finitely generated cohomology, it follows that $H^1(P;\mathbb{Z})$ is either $0$ or $\mathbb{Z}$.
\smallskip

{\it Claim.} $h$ admits an extension to a homeomorphism $\hat{h} : \mathbb{S}^3 - P \longrightarrow \mathbb{S}^3 - K$.

{\it Proof of claim.} Let $B \subseteq \mathbb{R}^3$ be a closed $3$--ball so big that it contains $P$ in its interior and denote by $S$ its boundary $2$--sphere. Clearly $\mathbb{R}^3 - S$ has two connected components: an unbounded one $U$ (which is the complement of $B$) and a bounded one ${\rm int}\ B$. Let us examine the connected components of $(\mathbb{R}^3 - P) - S$. Since $P$ has zero cohomology in degree $2$, it does not disconnect any connected open set that contains it (this follows from Alexander's duality). Writing $(\mathbb{R}^3 - P) - S = U \cup (({\rm int}\ B) - P)$ thus exhibits $(\mathbb{R}^3 - P) - S$ as the disjoint union of two connected sets, which are therefore its connected components. Notice that both have nonzero homology in dimension $2$: the first one because it is just homeomorphic to $\mathbb{S}^2 \times \mathbb{R}$ by construction; the second one because it is the result of removing a nonempty compact set from the interior of an open set with trivial homology (formally, this follows again from an application of Alexander's duality).

Since the $2$--sphere $S$ is contained in $\mathbb{R}^3 - P$ we may transform it via $h$ to obtain another $2$--sphere $S' := h(S)$. By the polyhedral Sch\"onflies theorem, $S'$ bounds a closed $3$--ball $B'$ in $\mathbb{R}^3$. Notice that we do not know a priori whether $K$ is contained in this ball; this is precisely what we want to prove now. Denote by $U' := \mathbb{R}^3 - B'$. The same reasoning used in the previous paragraph, but now taking into account that $K$ may be contained in $U'$ or in $B'$, leads to the conclusion that the connected components of $(\mathbb{R}^3 - K) - S'$ are either \[U' - K \quad \text{and} \quad {\rm int}\ B' \quad \quad \text{(if $K \subseteq U'$)},\] or \[U' \quad \text{and} \quad ({\rm int}\ B') - K \quad \quad \text{(if $K \subseteq B'$)}.\] We mentioned earlier that both components of $(\mathbb{R}^3 - P) - S$ have nonzero homology in dimension $2$, so the same must be true of the components of $(\mathbb{R}^3 - K) - S'$ because the two spaces are homeomorphic under $h$. This rules out the first alternative above, because ${\rm int}\ B'$ has trivial homology. Thus it must be the case that $K \subseteq {\rm int}\ B'$ and the components of $(\mathbb{R}^3 - K)-S'$ are precisely $U'$ and $({\rm int}\ B') - K$. Now, there are two possibilities for how $h$ matches the components of $(\mathbb{R}^3 - P)-S$ and $(\mathbb{R}^3 - K)-S'$. Either we have \[h(U) = U' \quad \text{and} \quad h(({\rm int}\ B) - P) = ({\rm int}\ B') - K\] or we have \[h(U) = ({\rm int}\ B') - K \quad \text{and} \quad h(({\rm int}\ B) - P) = U'.\] However, the second alternative is not possible. From $h(U) = ({\rm int}\ B') - K$ we would get $h(U \cup S) = B' - K$, and this would extend to a homeomorphism between the one--point compactifications of $U \cup S$ and $B' - K$. That of $U \cup S$ is just a $3$--ball, whereas the one--point compactification of $B' - K$ can be thought of as $B'/K$ with $[K]$ acting as the point at infinity. But then $B'/K$ is a $3$--ball, which implies that $K$ is pointlike and by \cite[Proposition 2.4.5, p. 61]{davermanvenema1} that $K$ is actually cellular, contradicting the definition of a toroidal set. We are left with $h(U) = U'$. Then $h(U \cup S) = U' \cup S$ and this can be extended to the point at infinity to yield the required $\hat{h}$. This concludes the proof of the claim. $_{\blacksquare}$
\smallskip

Let $N$ be a regular neighbourhood of $P$ and define $T := \hat{h}(N - P) \cup K$. Clearly $N$ contains $K$. In fact, the following holds:
\smallskip

{\it Claim.} $T$ is a compact manifold that is a neighbourhood of $K$. Moreover, $T$ is connected and $\partial T$ is either a $2$--sphere or a $2$--torus.

{\it Proof of claim.} Begin by writing $\mathbb{S}^3 = (\mathbb{S}^3 - N) \uplus (N - P) \uplus P$. Then we see that $\mathbb{S}^3 = \hat{h}(\mathbb{S}^3 - N) \uplus \hat{h}(N - P) \uplus K = \hat{h}(\mathbb{S}^3 - N) \uplus T$, and this gives the convenient relation $\mathbb{S}^3 - T = \hat{h}(\mathbb{S}^3 - N)$.

To prove that $N$ is a neighbourhood of $K$, suppose the contrary. Then there exists a sequence $\{q_n\}$ in $\mathbb{S}^3 - T$ that converges to some point in $K$. The sequence $\{p_n := \hat{h}^{-1}(q_n)\}$ is then contained in $\hat{h}^{-1}(\mathbb{S}^3 - T) = \mathbb{S}^3 - N$, and so after passing to a subsequence we may assume that it converges to some point $p \in \overline{\mathbb{S}^3 - N} = \mathbb{S}^3 - {\rm int}\ N$. But this set is still contained in the domain of $\hat{h}$, and so $q_n = \hat{h}(p_n)$ would converge to $\hat{h}(p)$, which is not possible because $\{q_n\}$ does not converge in $\mathbb{S}^3 - K$.

Writing $T = (T - K) \cup {\rm int}\ T$ exhibits it as the union of two open (in $T$) sets. The first is a $3$--manifold with boundary because it is homeomorphic to $N - P$ via $\hat{h}^{-1}$ by definition, whereas the second is a $3$--manifold without boundary because it is open in $\mathbb{S}^3$. Thus $T$ is a compact $3$--manifold with boundary. Resorting again to Alexander's duality and bearing in mind that $\mathbb{S}^3 - T$ and $\mathbb{S}^3 - N$ are homeomorphic via $\hat{h}$ we have \[H^q(T;\mathbb{Z}) = H_{3-q}(\mathbb{S}^3 - T;\mathbb{Z}) = H_{3-q}(\mathbb{S}^3 - N;\mathbb{Z}) = H^q(N;\mathbb{Z}) = H^q(P;\mathbb{Z}),\] where in the last step we have used that $N$ is a regular neighbourhood of $P$ and so collapses onto it. Recalling the computation of the cohomology of $P$ at the beginning of this proof we conclude that $T$ is connected, has vanishing cohomology in degrees $q \geq 2$, and has cohomology either $0$ or $\mathbb{Z}$ in dimension $q = 1$. A simple argument using Lefschetz duality then shows that $\partial T$ is either a $2$--sphere or a $2$--torus. $_{\blacksquare}$
\smallskip

Now we repeat the same construction not just for a single $N$, but for a neighbourhood basis of $P$ comprised of nested regular neighbourhoods $N_j$. As before, denote by $T_j = \hat{h}(N_j - P) \cup K$ the corresponding neighbourhoods of $K$. Since the sets $\mathbb{S}^3 - N_j$ clearly form an ascending sequence whose union is $\mathbb{S}^3 - P$, the sets $\hat{h}(\mathbb{S}^3 - N_j)$ also form an ascending sequence whose union is $\hat{h}(\mathbb{S}^3 - P) = \mathbb{S}^3 - K$. It then follows from the expression $\mathbb{S}^3 - T_j = \hat{h}(\mathbb{S}^3 - N_j)$ obtained above that the $T_j$ form a decreasing sequence whose intersection is $K$. Thus $\{T_j\}$ is a nested neighbourhood basis of $K$ comprised of compact, connected manifolds.

We established earlier that $\partial T_j$ is either a $2$--sphere or a $2$--torus. In the first case, the Sch\"onflies theorem for polyhedral spheres implies that $T_j$ is a $3$--ball. If this occurs for infinitely many indices $j$ then $K$ would be cellular, which contradicts the definition of a toroidal set. Thus we may assume without loss of generality that all the $\partial T_j$ are $2$--tori. Consider any one $T_{j_0}$. Since $K$ is toroidal, it has a neighbourhood $T$ which is a solid torus contained in the interior of $T_{j_0}$. Choose $j_1 > j_0$ so that $T_{j_1} \subseteq {\rm int}\ T_{j_0}$. Observe that $T_{j_0} - {\rm int}\ T_{j_1}$ is homeomorphic via $h$ to $N_{j_0} - {\rm int}\ N_{j_1}$, which is in turn homeomorphic to $(\partial N_{j_0}) \times [0,1]$ by the annulus theorem for regular neighbourhoods. Pulling back this homeomorphism via $h$ shows that $T_{j_0} - {\rm int}\ T_{j_1}$ is also homeomorphic to $(\partial T_{j_0}) \times [0,1]$. A direct application of a concentricity theorem of Edwards (\cite[Theorem 2, p. 419]{chedwards2}) then ensures that $T$ is concentric with both $T_{j_0}$ and $T_{j_1}$, so in particular all three of them are homeomorphic; hence, they are all (concentric) solid tori. Thus $\{T_j\}$ is the required neighbourhood basis of $K$.
\smallskip

(ii) $\Rightarrow$ (iii) Let $\{T_j\}$ and $\{T'_k\}$ be two neighbourhood bases of $K$ comprised of nested solid tori. Assume that the $T_j$ are all concentric. Choose $k_0$ big enough so that $T'_{k_0} \subseteq {\rm int}\ T_1$. Pick any $k_1 > k_0$ and finally let $j$ be big enough so that $T_{j} \subseteq {\rm int}\ T'_{k_1}$. Thus we have the nested tori $T_{j} \subseteq T'_{k_1} \subseteq T'_{k_0} \subseteq T_1$. Since $T_1$ and $T_{j}$ are concentric by assumption, the concentrity theorem of Edwards mentioned in the previous paragraph ensures that $T'_{k_0}$ is concentric with both $T_1$ and $T_j$. Then again the same theorem, this time applied to $T_j \subseteq T'_{k_1} \subseteq T'_{k_0}$, ensures that $T'_{k_1}$ is concentric with $T'_{k_0}$. Thus the $\{T'_k\}$ are all concentric for $k \geq k_0$.
\smallskip

(iii) $\Rightarrow$ (i) Let $\{T_j\}$ be a standard basis for $K$ comprised of concentric tori. Denote by $R_j$ each region $T_j - {\rm int}\ T_{j+1}$, so that $T_1 - K = \bigcup R_j$. Since the $\{T_j\}$ are concentric, each $R_j$ is homeomorphic to $(\partial T_j) \times [j,j+1]$ via some homeomorphism $h_j$ such that $h_j(p,j) = p$ for every $p \in \partial T_j$. It is then easy to modify the $h_j$ in such a way that they can all be pasted together to yield a homeomorphism from $\bigcup R_j$ onto $(\partial T_1) \times [1,+\infty)$. (See the proof of implication (ii) $\Rightarrow$ (i) in \cite[Theorem 3.11, p. 18]{hecyo1} for more details). The latter is, in turn, homeomorphic to $T_1 - \gamma$, where $\gamma$ is a core curve of $T_1$. Thus we have obtained a homeomorphism $T_1 - K \cong T_1 - \gamma$ that is the identity on $\partial T_1$. Extending this by the identity to all of $\mathbb{R}^3$ yields a homeomorphism between $\mathbb{R}^3 - K$ and $\mathbb{R}^3 - \gamma$, where $\gamma$ is a polyhedral simple closed curve. This shows that $K$ is weakly tame.
\end{proof}

\begin{remark} Since a polyhedral simple closed curve is perhaps the simplest example of a toroidal set, it would also seem reasonable to define a toroidal set to be weakly tame if its complement is homeomorphic to the complement of such a curve, rather than to the complement of an arbitrary polyhedron $P$. It follows from the proof of (iii) $\Rightarrow$ (i) in the preceding proposition that both definitions are equivalent.
\end{remark}

The next two theorems provide useful characterizations of weakly tame sets in terms of the genus and the self-index:

\begin{theorem} \label{teo:carac1} Let $K \subseteq \mathbb{R}^3$ be a toroidal set. Then $K$ is weakly tame if and only the genus of $K$ is finite and $\mathcal{N}(K) \sim 1$.
\end{theorem}

When the genus of $K$ is strictly positive, the condition $\mathcal{N}(K) \sim 1$ can be replaced with a weaker one:

\begin{theorem} \label{teo:carac2} Let $K \subseteq \mathbb{R}^3$ be a toroidal set with positive genus. Then $K$ is weakly tame if and only if the genus of $K$ is finite and $K$ is nontrivial.
\end{theorem}

\begin{example} \label{ex:knotted} (1) None of the (unknotted) generalized solenoids of Example \ref{ex:gen_sol} satisfies $\mathcal{N} \sim 1$; thus, none of them is weakly tame.

(2) Start with a nontrivial toroidal set $K \subseteq \mathbb{R}^3$ which satisfies $\mathcal{N}(K) \not\sim 1$. Let $T$ be a neighbourhood of $K$ that is a solid torus and such that the inclusion $K\subseteq T$ induces a nonzero map in cohomology. Let $e : T \longrightarrow \mathbb{R}^3$ be an embedding such that $e(T)$ is nontrivially knotted. Then $K' := e(K)$ has infinite genus.

{\it Proof.} Notice that $\check{H}^1(K';\mathbb{Z}) = \check{H}^1(K;\mathbb{Z}) \neq 0$ because \v{C}ech cohomology is a topological invariant; similarly $\mathcal{N}(K') = \mathcal{N}(K) \not\sim 1$ by the invariance of the self-index under local homeomorphisms (Proposition \ref{prop:local}). Thus $K'$ is not weakly tame by Theorem \ref{teo:carac1}. The genus of $K'$ is positive: if $K'$ were unknotted, then $e(T)$ should also be unknotted (\cite[Proposition 2.10, p. 10]{hecyo1}), but it is not by construction. It then follows from Theorem \ref{teo:carac2} that the genus of $K'$ must be infinite.
\end{example}

Now we aim to prove Theorems \ref{teo:carac1} and \ref{teo:carac2} above. The proof of the first is rather simple:

\begin{proof}[Proof of Theorem \ref{teo:carac1}] Suppose first $K$ is weakly tame. Then it has a standard basis of concentric tori $\{T_j\}$ by Proposition \ref{prop:charac_wtame}. Clearly concentric tori have the same genus and the geometric index $N(T_{j+1},T_j)$ is one. It follows that $g(K) < +\infty$ (by definition) and $\mathcal{N}(K) \sim 1$ by Proposition \ref{prop:indep}.

Now suppose $g(K) < +\infty$ and $\mathcal{N}(K) \sim 1$. By the definition of the genus of a toroidal set there exists a standard basis $\{T_j\}$ for $K$ such that $g(T_j) = g(K)$ for every $j$. Also, the assumption that $\mathcal{N}(K) \sim 1$ implies, by Proposition \ref{prop:N1}, that $N(T_{j+1},T_j) = 1$ for large enough $j$. This condition on the geometric index entails (see \cite[Satz 2, p. 171]{schubert}) that the core curve of $T_{j+1}$ is the connected sum of the core curve of $T_j$ with some other knot $\gamma_j$. But, since the genera of $T_j$ and $T_{j+1}$ are equal and genus is additive under connected sums, it follows that the genus of $\gamma_j$ must be zero. Therefore the core curves of $T_j$ and $T_{j+1}$ are equivalently knotted. A result of Edwards \cite[Theorem 3, p.4]{edwards1} then implies that $T_j$ and $T_{j+1}$ are concentric. Hence $K$ is weakly tame by Proposition \ref{prop:charac_wtame}.
\end{proof}

The proof of Theorem \ref{teo:carac2} is slightly more involved and requires a previous lemma:

\begin{lemma} \label{lem:converse} Let $T_2 \subseteq T_1$ be a nested pair of solid tori. Assume that:
\begin{itemize}
	\item[(i)] The winding number $w$ of $T_2$ inside $T_1$ is positive.
	\item[(ii)] The genera $g$ of $T_1$ and $T_2$ are equal and positive.
\end{itemize}
Then $T_1$ and $T_2$ are concentric.
\end{lemma}

Before proving the lemma, recall that a spine of a polyhedral manifold $N$ is a polyhedron $P \subseteq {\rm int}\ N$ such that $N$ collapses onto $P$, which we denote by $N \searrow P$ as customary. We will need the following result of Hudson and Zeeman (see \cite[Corollary 5, p. 727]{hudsonzeeman2}): if $P$ and $P'$ are two polyhedra in ${\rm int}\ N$ that are related by a sequence of collapses and expansions in $N$ (but not necessarily in ${\rm int}\ N$) then $P$ is a spine of $N$ if and only if $P'$ is a spine of $N$. In our case $N$ will always be a neighbourhood of $P$ in the ambient $3$--manifold (so in particular its interior as a manifold coincides with its topological interior), and then saying that $P$ is a spine of $N$ is equivalent to saying that $N$ is a regular neighbourhood of $P$ (see \cite[Corollary 3.30, p. 41]{rourkesanderson1}).

\begin{proof}[Proof of Lemma \ref{lem:converse}] Let $\gamma_1$ and $\gamma_2$ be core curves for $T_1$ and $T_2$ respectively. Let also $\lambda_1 \subseteq \partial T_1$ be a longitude of $T_1$.
\smallskip

{\it Claim.} It suffices to show that $\gamma_2$ and $\lambda_1$ cobound an annulus in $T_1$.

{\it Proof of claim.} Let $A'$ be the annulus cobounded by $\gamma_2$ and $\lambda_1$ in $T_1$, and denote by $A$ an annulus cobounded by $\gamma_1$ and $\lambda_1$ in $T_1$ (this always exists). Any annulus collapses onto each one of its boundary curves. Thus we may write $\gamma_1 \nearrow A \searrow \lambda_1 \nearrow A' \searrow \gamma_2$ within $T_1$, and since $\gamma_1$ is a spine of $T_1$, the result of Hudson and Zeeman mentioned earlier implies that $\gamma_2$ is a spine of $T_1$. This means that $T_1$ is a regular neighbourhood of $\gamma_2$ in $\mathbb{R}^3$ and, since $T_2$ is also a regular neighbourhood of $\gamma_2$ (the latter being a core for $T_2$), the regular neighbourhood annulus theorem (see for instance \cite[Corollary 1, p. 725]{hudsonzeeman2} or \cite[Corollary 3.18, p. 36]{rourkesanderson1}) guarantees that $T_2$ and $T_1$ are concentric. This proves the claim. $_{\blacksquare}$
\smallskip

Now we shall show $\gamma_2$ and $\lambda_1$ indeed cobound an annulus in $T_1$. This follows from an examination of the proof of Schubert's relation between the genus of a satellite knot and its companion as given, for example, in the book by Burde and Zieschang \cite{burdezieschang1}. In our setting $\gamma_2$ is a satellite of $\gamma_1$ and the relation just mentioned reads $g(T_2) \geq w \cdot g(T_1) + g'$, where $g'$ is the genus of the pattern of $T_2$ inside $T_1$. By assumption $g(T_1) = g(T_2) = g > 0$ and (since $g' \geq 0$ and $w \geq 1$), this inequality implies $w = 1$.

Let $S$ be a Seifert surface of minimal genus that spans $\gamma_2$. Without loss of generality we may assume that $S$ is transversal to $\partial T_1$. Let $S_i$ and $S_o$ be the parts of $S$ that lie inside and outside $T_1$ respectively; that is, $S_i := S \cap T_1$ and $S_o := S \cap \overline{(\mathbb{R}^3 - T_1)}$. Both of these are (a priori, possibly nonconnected) orientable surfaces with boundary. We may choose (see \cite[Lemma 2.11, p. 21]{burdezieschang1}) $S$ so that: (i) it intersects $\partial T_1$ transversally in a longitude $\lambda_1$, and (ii) $S_o$ consists of a single connected component (whose boundary is, therefore, precisely $\lambda_1$).

Since $\lambda_1$ is a longitude of $T_1$, it cobounds an annulus in $T_1$ with its core curve $\gamma_1$. The union of this annulus and $S_o$ produces a Seifert surface for $\gamma_1$ whose genus is the same as that of $S_o$. Thus the genus of $S_o$ is bigger than or equal to the genus of $\gamma_1$, namely $g$. In terms of the Euler characteristic, $\chi(S_o) \leq 1 - 2g$.

Recall that $S$ was a minimal Seifert surface for $\gamma_2$, which also has genus $g$ by assumption. Thus we have $\chi(S) = 1-2g$. Then from $\chi(S) = \chi(S_i) + \chi(S_o)$ and the inequality of the previous paragraph we get $\chi(S_i) \geq 0$. Since $S_i$ arose by cutting the connected surface $S$ along $\partial T_1$, the boundary of each component of $S_i$ must have a nonempty intersection with $\partial T_1$. But, since $S \cap \partial T_1$ consists of the single curve $\lambda_1$, it follows that $S_i$ is actually connected and its boundary has precisely two components; namely $\gamma_2$ and $\lambda_1$. The only connected, orientable surface with two boundary components and nonnegative Euler characteristic is the annulus. Thus $S_i$ is the required annulus that cobounds $\gamma_2$ and the longitude $\lambda_1$ in $T_1$.
\end{proof}

\begin{proof}[Proof of Theorem \ref{teo:carac2}] It follows from the definition of the genus that $K$ has a standard basis $\{T_j\}$ such that $g(T_j) = g(K)$ for all $j$. Denote by $w_j$ the winding number of $T_{j+1}$ inside $T_j$. Since $\check{H}^1(K;\mathbb{Z}) \neq 0$ by assumption, we must have $w_j > 0$ for big enough $j$. It then follows from Lemma \ref{lem:converse} that all the pairs $(T_j,T_{j+1})$ are concentric from some $j$ onwards. Finally, Proposition \ref{prop:charac_wtame} implies that $K$ is weakly tame.
\end{proof}

\section{Toroidal sets with prescribed cohomology and self-index} \label{sec:construct}

Recall that one of the goals of this paper is to exhibit examples of toroidal sets having finite genus that cannot be realized as attractors for homeomorphisms of $\mathbb{R}^3$. To this end it is convenient to have an ample catalogue of toroidal sets. We will be especially interested in those that are not weakly tame, since a weakly tame set can always be realized as an attractor for a flow.

In this section we show how to construct toroidal sets $K$ that have a prescribed cohomology group $H$ (in degree $1$) and whose self-index is some prescribed group $N$. Of course $H$ and $N$ cannot be entirely arbitrary since they must be the direct limit of a direct sequence of the form \eqref{eq:dir_lim_Z}. Let us say that a group is \emph{feasible} if it has this form. Also, the fact that toroidal sets cannot be cellular requires that $N \neq 0$ by Proposition \ref{prop:cellular}. Since we are interested in toroidal sets that are not weakly tame and have finite genus, we must have $N \not\sim 1$ by Theorem \ref{teo:carac1}. Finally, the variables $H$ and $N$ are not entirely independent: in Proposition \ref{prop:2} we saw that the consistency condition has to be satisfied that $2 | H$ if and only if $2 | N$. The following theorem shows that for unknotted toroidal sets there are no further constraints:

\begin{theorem} \label{teo:construct} Let $H$ and $N$ be feasible groups. Suppose that $N \neq 0$, $N \not\sim 1$, and suppose the following consistency condition is satisfied: $2 | H$ if and only if $2 | N$. Then there exists a toroidal set $K$ such that:
\begin{itemize}
	\item[(i)] $\check{H}^1(K;\mathbb{Z}) = H$.
	\item[(ii)] $\mathcal{N}(K) = N$.
	\item[(iii)] $K$ is unknotted.
\end{itemize}
\end{theorem}

\begin{corollary} \label{cor:construct} Let $H$ be a feasible group and let $P$ be a set of prime numbers. Assume that the consistency condition holds: $2 \in P$ if and only if $2 | H$. Then there exists a toroidal set $K$ such that:
\begin{itemize}
	\item[(i)] $\check{H}^1(K;\mathbb{Z}) = H$.
	\item[(ii)] The prime divisors of $\mathcal{N}(K)$ are precisely the elements of $P$.
	\item[(iii)] $K$ is unknotted.
\end{itemize}
\end{corollary}
\begin{proof} Apply the preceding theorem with $N$ as constructed in Example \ref{ex:p_div}.
\end{proof}

For genus $g > 0$, and recalling again that we are interested in sets that are not weakly tame, Theorem \ref{teo:carac2} requires that $H = 0$. In particular $2 | H$, and so we must also require that $2 |N$. Our understanding of this case is much more limited, and we have only been able to produce examples with genus $g = 1$:

\begin{theorem} \label{teo:construct_trivial} Let $N \neq 0$ be a feasible group such that $2 | N$. Then there exists a toroidal set $K$ such that:
\begin{itemize}
	\item[(i)] $K$ is trivial.
	\item[(ii)] $\mathcal{N}(K) = N$.
	\item[(iii)] The genus of $K$ is $1$.
\end{itemize}
\end{theorem}

\begin{corollary} \label{cor:construct_trivial} Let $P$ be a set of prime numbers that contains $2$. Then there exists a toroidal set $K$ such that:
\begin{itemize}
	\item[(i)] $K$ is trivial.
	\item[(ii)] The prime divisors of $\mathcal{N}(K)$ are precisely the elements of $P$.
	\item[(iii)] The genus of $K$ is $1$.
\end{itemize}
\end{corollary}
\begin{proof} Apply the preceding theorem with $N$ as constructed in Example \ref{ex:p_div}.
\end{proof}

To prove these theorems we need an auxiliary lemma which provides the key to the construction:

\begin{lemma} \label{lem:torus} Let $w$ and $k$ be a pair of nonnegative integers. Denote by $V_0 \subseteq \mathbb{R}^3$ the standard unknotted solid torus. Then there exists an unknotted solid torus $V_1 \subseteq {\rm int}\ V_0$ such that the winding number of $V_1$ inside $V_0$ is $w$ and the geometric index $N(V_1,V_0)=w+2k$. 
\end{lemma}
\begin{proof} To prove the lemma it suffices to find an unknotted simple closed curve $\gamma$ in ${\rm int}\ V_0$ such that its winding number and geometric index in $V_0$ are $w$ and $w + 2k$ respectively, since then any regular neighbourhood $V_1$ of $\gamma$ contained in ${\rm int}\ V_0$ will fulfill our requirements.

Figure \ref{fig:whitehead}.(a) shows a view of $V_0$ from the top and, contained in its interior, two groups of oriented simple closed curves. The outermost group, labeled ({\bf 1}), consists of $w$ curves $\gamma_1,\ldots,\gamma_w$ that wind monotonically once around $V_0$. The innermost group ({\bf 2}) consists of $k$ curves $\gamma_{w+1},\ldots,\gamma_{w+k}$ which we call \emph{Whitehead curves} because they are patterned after one of the components of the Whitehead link. All the curves in the first group have the same orientation. The Whitehead curves have alternating orientations, chosen in such a way that the outermost strand of the outermost Whitehead curve runs parallel to the innermost curve of group ({\bf 1}).

\begin{figure}[h!]
\null\hfill
\subfigure[Initial setup]{
\begin{pspicture}(0,0)(6.6,6.6)
	\rput[bl](0,0){\scalebox{0.8}{\includegraphics{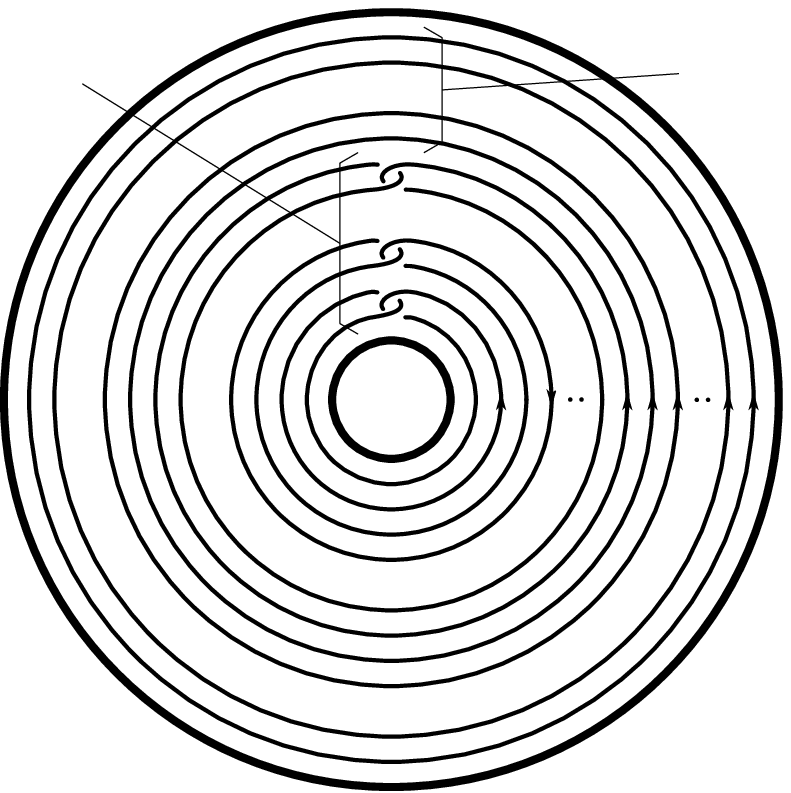}}}
	\rput[bl](5.2,0.2){$V_0$}
	\rput[l](5.6,5.8){\tiny{({\bf 1})}}
	\rput[r](0.6,5.8){\tiny{({\bf 2})}}
\end{pspicture}}
\hfill
\subfigure[Final setup]{
\begin{pspicture}(0,0)(6.6,6.6)
	\rput[bl](0,0){\scalebox{0.8}{\includegraphics{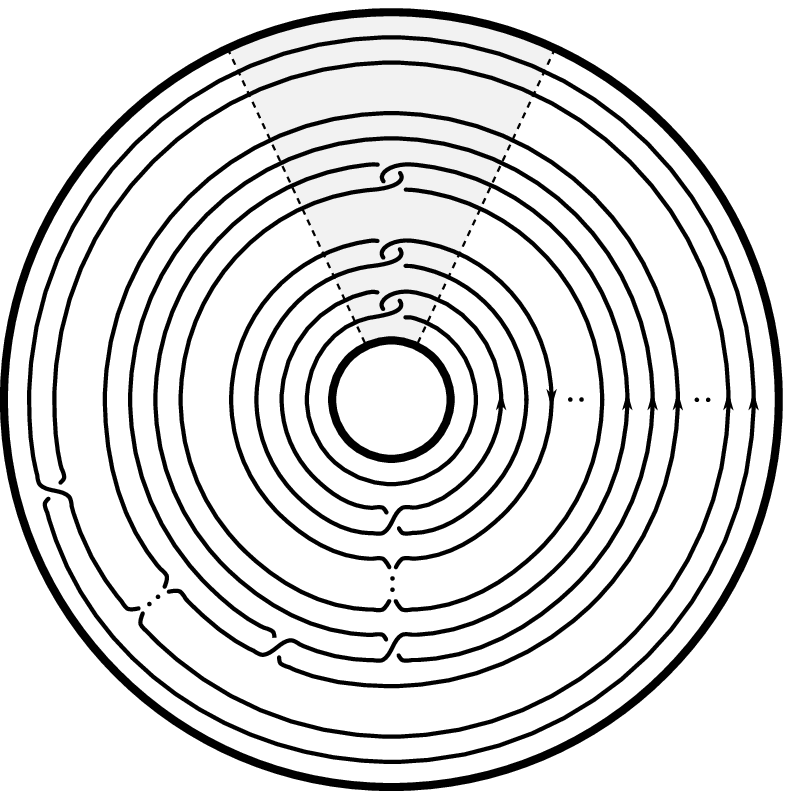}}}
	\rput[bl](5.2,0.2){$V_0$}
\end{pspicture}}
\hfill\null
\caption{ \label{fig:whitehead}}
\end{figure}

Let $\gamma$ be the result of taking the (oriented) connected sum of all the $\gamma_j$ arranged as shown in Figure \ref{fig:whitehead}.(b). The choice of orientations ensures that in performing the connected sum of each curve with the adjacent one we include a twist rather than having a turning point where $\gamma$ would double back over.

Since each $\gamma_j$ is the unknot, the same is true of $\gamma$. It is also clear from the construction that each $\gamma_j$ with $j \leq w$ has winding number $1$ in $V_0$ while each $\gamma_j$ with $j>w$ has winding number $0$, so that $\gamma$ has winding number $w$ in $V_0$. It only remains to check that the geometric index of $\gamma$ in $V_0$ is precisely $2k+w$. To see this we apply a technique of Andrist, Garity, Repov\v{s} and Wright \cite{and-gar-rep}. The dotted radial lines in Figure \ref{fig:whitehead}.(b) represent two meridional disks which decompose the solid torus $V_0$ into two sectors $C_1$ and $C_2$, where $C_1$ is shaded light gray. Sectors of this form are called chambers in \cite{and-gar-rep}. For $1 \leq j \leq w$ the intersection $\gamma_j \cap C_1$ just consists of an arc that runs from one of the meridional disks to the other without turning back over. Arcs of this form are called spanning arcs. For $w+1 \leq j \leq w+k$ the intersection $\gamma_j \cap C_1$ consists of a pair of linked arcs; these are called Whitehead clasps (see \cite[Figure 5.(b), p. 9]{and-gar-rep}). The intersection of every $\gamma_j$ with the chamber $C_2$ just consists of one or two spanning arcs. In this situation \cite[Corollary 4.6, p. 237]{and-gar-rep} ensures that the geometric index $N(\gamma,V_0)=w+2k$.
\end{proof}

\begin{proof}[Proof of Theorem \ref{teo:construct}] Since $H$ and $N$ are feasible, both can be written as direct limits of the form \[H = \varinjlim\ \{ \xymatrix{\mathbb{Z} \ar[r]^{w_j} & \mathbb{Z}} \} \quad \text{ and } \quad N = \varinjlim\ \{ \xymatrix{\mathbb{Z} \ar[r]^{n_j} & \mathbb{Z}}\}\] where the $\{w_j\}$ and $\{n_j\}$ are as described in Lemma \ref{lem:canonical}. Of the cases considered there we must have $n_j \geq 2$ for all $j$ because $N \neq 0$ and $N \not\sim 1$ by assumption. 

Suppose first that $ 2 | H$. By Proposition \ref{prop:divisors} we see that $w_j$ must be even for infinitely many $j$, and by grouping the $w_j$ into blocks each of which contains at least an even arrow we may assume that all the $w_j$ are even. By the consistency condition of the theorem $2 | N$ too, and by the same argument we may take $n_j$ to be even for all $j$. A similar reasoning applies when $2 \not| H$ (this time discarding a finite number of the $w_j$ and $n_j$ instead of grouping arrows as before), and in that case we may take all the $w_j$ and $n_j$ to be odd. Summing up, we can assume without loss of generality that all the $w_j$ and $n_j$ have the same parity.

Now we claim that we may also assume that $w_j \leq n_j$ for every $j$. For, consider $w_1$. Since $n_j \geq 2$ for every $j$, by taking a product $n_1 \cdot \ldots \cdot n_{k_1}$ for sufficiently large $k_1$ we can achieve $w_1 \leq n_1 \cdot \ldots \cdot n_{k_1}$. Similarly, taking $k_2 > k_1$ big enough we can achieve $w_2 \leq n_{k_1+1} \cdot \ldots \cdot n_{k_2}$, and so on. That is, by grouping the arrows in the defining sequence for $N$ we can certainly assume that $w_j \leq n_j$ for every $j$. Notice that this does not change the parity of the $n_j$.

Our choice of the $w_j$ and $n_j$ ensures that the equation $2k_j + w_j = n_j$ has an integer solution $k_j \geq 1$ for each $j$. Then by Lemma \ref{lem:torus} we may construct a nested sequence of unknotted solid tori $\{T_j\}$ such that the winding number and the geometric index of $T_{j+1}$ inside $T_j$ are precisely $w_j$ and $n_j$ respectively. Let $K := \bigcap_{j \geq 0} T_j$. Proposition \ref{prop:cellular} shows that $K$ is not cellular; hence, it is a toroidal set. It is clear that $K$ is unknotted and has $\check{H}^1(K;\mathbb{Z}) = H$ and $\mathcal{N}(K) = N$ as required.
\end{proof}

\begin{proof}[Proof of Theorem \ref{teo:construct_trivial}] Write $N$ in the form \[N = \varinjlim\ \{\xymatrix{\mathbb{Z} \ar[r]^{\cdot n_j} & \mathbb{Z}}\}.\]  Since $2|N$, arguing as in the proof of Theorem \ref{teo:construct} we may assume that all the $n_j$ are even; moreover, since $N \neq 0$ we may also assume $n_j \geq 2$ for every $j$ (see Lemma \ref{lem:canonical}). 

We construct the required toroidal set as the intersection of a nested family of tori $\{T_j\}$ as follows. Let $V$ be the standard unknotted torus in $\mathbb{R}^3$ and let $W \subseteq {\rm int}\ V$ be another solid torus arranged as a Whitehead curve (that is, as the innermost curve in Figure \ref{fig:whitehead}.(a)). Now start with a torus $T_1 \subseteq \mathbb{R}^3$ knotted in a nontrivial way. Let $W_1 \subseteq {\rm int}\ T_1$ be another solid torus which lies inside $T_1$ according to the pattern $(V,W)$. Now let $T_2 \subseteq {\rm int}\ W_1$ be a third solid torus that winds monotonically $n_1/2$ times inside $W_1$ (as in the generalized solenoids). Let again $W_2 \subseteq {\rm int}\ T_2$ be a solid torus which lies inside $T_2$ in the pattern $(V,W)$ and then let $T_3 \subseteq {\rm int}\ W_2$ wind monotonically $n_2/2$ times inside $W_2$. Continue in this fashion and define $K := \bigcap_j T_j$. Notice that both $\{T_j\}$ and $\{W_j\}$ are standard bases for $K$.

Since the winding number of $W$ inside $V$ is zero, the same is true of every pair $W_j \subseteq T_j$. Thus $\check{H}^1(K;\mathbb{Z}) = 0$. Notice that $N(T_{j+1},T_j) = N(T_{j+1},W_j) \cdot N(W_j,T_j) = (n_j/2) \cdot 2 = n_j$ by construction. This implies that $K$ is indeed a toroidal set (because it is not cellular by Proposition \ref{prop:indep}) and also that $\mathcal{N}(K) = N$ by Proposition \ref{prop:indep}.

Finally we show that $g(K) = 1$. This is very similar to the proof that the Whitehead double of a knot has genus one. First, since the core curve of $W$ bounds a Seifert surface of genus $1$ inside the torus $V$, each $W_j$ also bounds a Seifert surface of genus $1$ and therefore $g(K) \leq 1$ because the $W_j$ form a neighbourhood basis of $K$. To prove that $g(K) > 0$, suppose on the contrary that $K$ were unknotted. Then $K$ would have a neighbourhood basis of unknotted tori, and in particular we could interpolate an unknotted torus $U$ between $T_1$ and $T_j$ for some big enough $j$. Then $N(T_j,T_1) = N(T_j,U) \cdot N(U,T_1)$ and, since $N(T_j,T_1) = N(T_j,T_{j-1}) \cdot \ldots \cdot N(T_2,T_1) = n_{j-1} \cdot \ldots \cdot n_1 \neq 0$ by the computations of the previous paragraph, $N(U,T_1)$ is also nonzero. In particular $U$ is a satellite of $T_1$, but this is impossible: the unknot cannot be a satellite of a nontrivial knot (see \cite[Remark before Proposition 3.12, p. 39]{burdezieschang1} or \cite[Corollary 10, p. 113]{rolfsen1}).
\end{proof}

\section{Application to dynamics: the realizability problem} \label{sec:dyn}

We finally turn to dynamics and the realizability problem for toroidal sets: finding a topological characterization of those toroidal sets $K \subseteq \mathbb{R}^3$ that can be realized as a (local) attractor for a flow or a homeomorphism of $\mathbb{R}^3$. In the discrete case we will only obtain partial results which, however, provide a substantial clarification of the analysis in \cite{hecyo1}. 

Let us recall some standard definitions from dynamics. We state them for the discrete case, but the case of flows is completely ananlogous. Let $f$ be a homeomorphism of $\mathbb{R}^3$. An \emph{attractor} for $f$ is a compact invariant set $K \subseteq \mathbb{R}^3$ which is stable in the sense of Lyapunov and attracts all points in some neighbourhood $U$ of $K$. The maximal neighbourhood $U$ with this property is called the \emph{basin of attraction} of $K$ and is an open invariant subset of $\mathbb{R}^3$. We denote it by $\mathcal{A}(K)$. A stable attractor not only attracts every point in its basin of attraction, but also every compact subset of it. More precisely, for every compact set $P \subseteq \mathcal{A}(K)$ and every neighbourhood $V$ of $K$ there exists $n_0$ such that $f^n(P) \subseteq V$ for every $n \geq n_0$. It is in this form that we will make use of the stability assumption.

We mentioned in the Introduction that the realizability problem for flows has a solution in terms of weak tameness. For toroidal sets, being weakly tame is equivalent to having finite genus and $\mathcal{N} \sim 1$ as established in Theorem \ref{teo:carac1}. Thus we have:

\begin{theorem} \label{teo:flows} A toroidal set $K \subseteq \mathbb{R}^3$ can be realized as an attractor for a flow if and only if it has finite genus and $\mathcal{N}(K) \sim 1$.
\end{theorem}

\begin{remark} An alternative (and somewhat more elementary) proof can be obtained by resorting to \cite[Theorem 3.11, p.19]{hecyo1}, where it is shown that a toroidal set can be realized as an attractor for a flow if and only if it has a neighbourhood basis of concentric solid tori. This is equivalent to requiring that $K$ be weakly tame by Proposition \ref{prop:charac_wtame}.
\end{remark}

Now we turn to the case of homeomorphisms, where the realizability problem is much more difficult. In fact we cannot solve it completely, but by using the self-index we will be able to improve our understanding of it somewhat. Our main tool will be the following theorem:

\begin{theorem} \label{teo:main1} Let $K$ be a toroidal set that is an attractor for a homeomorphism $f$ of $\mathbb{R}^3$. Then $\mathcal{N}(K)$ is number-like. Moreover, $\mathcal{N}(K) \sim 1$ if and only if $K$ can be realized as an attractor for a flow.
\end{theorem}

Before proving the theorem we include some corollaries. A result of G\"unther \cite[Theorem 1, p. 653]{gunther1} shows that a generalized solenoid constructed in the manner described in Example \ref{ex:gen_sol}.(iii) cannot be realized as an attractor of a homeomorphism (in fact, of an arbitrary continuous map) of $\mathbb{R}^3$. Our first corollary sharpens this:

\begin{corollary} \label{cor:gen_sol_noatt} If a generalized solenoid can be realized as an attractor for a homeomorphism, then it must be (homeomorphic to) an $n$--adic solenoid for some $n$.
\end{corollary}
\begin{proof} Let $K$ be a generalized solenoid that is an attractor for a homeomorphism. By Theorem \ref{teo:main1} its self-index is number-like; say $\mathcal{N}(K) \sim n$ for some $n$. The self-index of an $n$--adic solenoid $K'$ also satisfies $\mathcal{N}(K') \sim n$, so by Proposition \ref{prop:n_like} the groups $\mathcal{N}(K)$ and $\mathcal{N}(K')$ are isomorphic. Then Proposition \ref{prop:charac_gsol} implies that $K$ and the $n$--adic solenoid $K'$ are homeomorphic.
\end{proof}

Gathering some of our results we can now see that the realizability problem for homeomorphisms and flows are equivalent for certain classes of toroidal sets:

\begin{corollary} \label{cor:equivalence} The realizability problem for homeomorphisms and flows are equivalent for:
\begin{itemize}
	\item[(i)] Toroidal sets whose self-index has no prime factors.
	\item[(ii)] Nontrivial toroidal sets with positive genus.
\end{itemize}
\end{corollary}
\begin{proof} (i) If $K$ is an attractor for a homeomorphism, then $\mathcal{N}(K)$ is number-like. Since it does not have prime divisors, it must be $\mathcal{N}(K) \sim 1$. The result follows from Theorem \ref{teo:main1}.

(ii) If $K$ can be realized as an attractor for a homeomorphism of $\mathbb{R}^3$ then it has finite genus by \cite[Theorem 3.1, p. 13]{hecyo1} and then Theorem \ref{teo:carac2} guarantees that it is weakly tame.
\end{proof}

Of course, in general the realizability problems for flows and homeomorphisms are not equivalent:

\begin{example} Consider an unknotted solid torus $V_0 \subseteq \mathbb{R}^3$ and let $f : \mathbb{R}^3 \longrightarrow \mathbb{R}^3$ be a homeomorphism such that the solid torus $f(V_0)$ lies inside the interior of $V_0$ in the pattern of a Whitehead curve (the innermost curve in Figure \ref{fig:whitehead}). The positively invariant set $V_0$ contains the attractor $K := \bigcap_{j \geq 0} f^j(V_0)$, which is called the Whitehead continuum. The tori $T_j := f^j(T)$ form a neighbourhood basis for $K$ and, since $N(T_{j+1},T_j) = 2$ for each $j$, one has that $K$ is a toroidal set whose self-index $\mathcal{N}(K)$ is the dyadic rationals. In particular $K$ cannot be realized as an attractor for a flow by Theorem \ref{teo:flows}. (This fact was established by different means in \cite[Example 47, p. 3622]{mio6}). 
\end{example}

As mentioned in the Introduction, if a toroidal set $K$ can be realized as an attractor for a homeomorphism of $\mathbb{R}^3$ then its genus $g(K)$ must be finite (see \cite[Theorem 3.1, p. 13]{hecyo1}). Combining Corollary \ref{cor:equivalence} with the characterization of Theorem \ref{teo:flows} we can obtain the following partial converses of this result:

\begin{corollary} Let $K \subseteq \mathbb{R}^3$ be a nontrivial toroidal set with positive genus. Then $K$ can be realized as an attractor for a homeomorphism if and only if its genus is finite.
\end{corollary}

\begin{corollary} Let $K \subseteq \mathbb{R}^3$ be a toroidal set such that $\mathcal{N}(K)$ has no prime factors. Then $K$ can be realized as an attractor for a homeomorphism if and only if its genus is finite and $\mathcal{N}(K) \sim 1$.
\end{corollary}

Now we turn to the proof of Theorem \ref{teo:main1}. Since the self-index is defined in terms of polyhedral solid tori, it is convenient to be able to replace the purely topological situation of the theorem with a piecewise linear one. This is accomplished by the following lemma:

\begin{lemma} \label{lem:attapprox} Let $K \subseteq \mathbb{R}^3$ be an attractor for a homeomorphism $f$ of $\mathbb{R}^3$. Then $K$ can be realized as an attractor for a homeomorphism of $\mathbb{R}^3$ that is piecewise linear on $\mathbb{R}^3 - K$.
\end{lemma}

In the proof of the lemma we will use the following notation: if $A \subseteq \mathbb{R}^3$ is a compact set and $\phi > 0$ is a positive number, we write $A + \phi$ to denote the set of points whose distance to $A$ is less than or equal to $\phi$. Obviously $A + \phi$ is also a compact set.

\begin{proof}[Proof of Lemma \ref{lem:attapprox}] Since $K$ is a stable attractor, it has a compact neighbourhood $P$ contained in its basin of attraction and such that $f(P)$ is contained in the interior of $P$. Define $P_k := f^k(P)$ for $k \geq 0$. These sets form a decreasing neighbourhood basis for $K$. Also, since each $P_{k+1}$ is contained in the interior of $P_k$, there exists a positive number $\phi_k > 0$ such that $P_{k+1} + \phi_k \subseteq P_k$. Clearly the $\phi_k$ can be chosen so that the sequence $\{\phi_k\}$ is strictly decreasing and converges to zero.

Let $U := \mathbb{R}^3 - K$. Notice that $U$ is the union of the sets $(P_k - P_{k+1})$ together with $\mathbb{R}^3 - P_0$. Define a (noncontinuous) mapping $\phi$ on $U$ by \[\phi(p) := \left\{ \begin{array}{ll} \phi_1 & \text{ if } p \in \mathbb{R}^3 - P_0, \\ \phi_{k+1} & \text{ if } p \in P_k - P_{k+1} \ (k \geq 0). \end{array} \right.\] It is easy to check (using that $\{\phi_k\}$ is a decreasing sequence) that $\phi|_{\mathbb{R}^3 - P_k}$ is bounded below by $\phi_{k+1}$. Since every compact subset of $U$ is contained in a set of the form $\mathbb{R}^3 - P_k$, it follows that $\phi$ is bounded away from zero on every compact subset of $U$. Thus $\phi$ is strictly positive in the sense of Moise \cite[p. 46]{moise2}. The map $f^2$ is a homeomorphism of $U$, and by the same approximation theorem that we invoked in the proof of Proposition \ref{prop:local} (namely \cite[Theorem 1, p. 253]{moise2}) there exists a piecewise linear homeomorphism $g$ of $U$ that is a $\phi$--aproximation of $f^2$; that is, such that ${\rm d}(f^2(p),g(p)) < \phi(p)$ for every $p \in U$.

Extend $g$ to a map $\hat{g}$ on all of $\mathbb{R}^3$ by defining $\hat{g} = f^2$ on $K$ (and of course $\hat{g} = g$ on $U$). Clearly $\hat{g}$ is a bijection of $\mathbb{R}^3$ and it is continous on $U$. We claim that $\hat{g}$ is also continuous on each $p \in K$. To check this pick $p \in K$ and a sequence $\{p_n\}$ in $U$ converging to $p$. By definition $\hat{g}(p) = f^2(p)$ and $\hat{g}(p_n) = g(p_n)$, so we may write \[{\rm d}(\hat{g}(p),\hat{g}(p_n)) = {\rm d}(f^2(p),g(p_n)) \leq {\rm d}(f^2(p),f^2(p_n)) + {\rm d}(f^2(p_n),g(p_n)) < {\rm d}(f^2(p),f^2(p_n)) + \phi(p_n)\] where in the last step we have used that $g$ is a $\phi$--approximation to $f^2$ on $U$. Now notice that $\{p_n\}$ eventually enters every $P_k$ and so $\phi(p_n)$ eventually becomes less than or equal to $\phi_{k+1}$ (we are again using the fact that $\{\phi_k\}$ is decreasing). Thus $\{\phi(p_n)\}$ converges to zero because the $\{\phi_k\}$ were chosen to converge to zero. Similarly ${\rm d}(f^2(p),f^2(p_n))$ converges to zero because $f^2$ is continuous. Hence the distance between $\hat{g}(p)$ and $\hat{g}(p_n)$ converges to zero and so $\hat{g}$ is continuous at $p$. In sum, $\hat{g}$ is a continuous bijection of $\mathbb{R}^3$ and therefore (by the invariance of domain theorem) it is a homeomorphism. Evidently it is piecewise linear on $\mathbb{R}^3 - K$ by construction.

To prove the lemma it only remains to show that $K$ is an attractor for $\hat{g}$. Since $g$ is a $\phi$--approximation to $f^2$ and $\phi|_{P_k - K} \leq \phi_{k+1}$, we have $g(P_k - K) \subseteq f^2(P_k - K) + \phi_{k+1}$. Using $f^2(P_k) = P_{k+2}$ we may then write \[g(P_k - K) \subseteq f^2(P_k - K) + \phi_{k+1} \subseteq P_{k+2} + \phi_{k+1} \subseteq P_{k+1}.\] Consequently $\hat{g}(P_k)$, which is the union of $g(P_k - K)$ and $\hat{g}(K) = f^2(K) = K$, is also contained in $P_{k+1}$. It follows that $K$ is an attractor for $\hat{g}$.
\end{proof}

\begin{proof}[Proof of Theorem \ref{teo:main1}] By Lemma \ref{lem:attapprox} we may assume that the homeomorphism $f$ which realizes $K$ as an attractor is piecewise linear on $\mathbb{R}^3 - K$. Denote by $\mathcal{A}(K)$ be the basin of attraction of $K$. Since $K$ is toroidal, it has a neighbourhood $T$ that is a polyhedral solid torus contained in $\mathcal{A}(K)$. Let $n$ be a sufficiently high iterate of $f$ so that $f^n(T)$ is contained in the interior of $T$ and define $T_j := f^{n j}(T)$ for $j \geq 0$. These $\{T_j\}$ form a nested family of polyhedral solid tori that provide a neighbourhood basis for $K$.

Now consider any pair $T_{j+1} \subseteq T_j$. The map $f^{nj}$ provides a homeomorphism from the pair $(T_0,T_1)$ onto the pair $(T_j,T_{j+1})$ which is piecewise linear on $T_0 - {\rm int}\ T_1$. Thus by Lemma \ref{lem:rel_hom} we have that $N(T_{j+1},T_j) = N(T_1,T_0)$. It then follows from Proposition \ref{prop:indep} that $\mathcal{N}(K)$ is number-like.

To conclude the proof observe that $K$ has finite genus because all the $T_j$ are ambient homeomorphic to each other and therefore have the same genus, and then Theorem \ref{teo:carac1} shows that $\mathcal{N}(K) \sim 1$ if and only if $K$ is weakly tame. But the latter is equivalent to $K$ being realizable as an attractor for a flow.
\end{proof}

Summing up, so far the panorama (organized in terms of the genus) is the following: given a toroidal set $K \subseteq \mathbb{R}^3$,
\begin{itemize}
	\item[(i)] If $g(K) = +\infty$, then $K$ cannot be realized as an attractor for a homeomorphism of $\mathbb{R}^3$.
	\item[(ii)] If $0 < g(K) < +\infty$ and $K$ is nontrivial, then it can be realized as an attractor for a flow and therefore also for a homeomorphism of $\mathbb{R}^3$.
	\item[(iii)] If $g(K) = 0$ (that is, $K$ is unknotted) and $\mathcal{N}(K) \sim 1$, then it can be realized as an attractor for a flow and therefore also for a homeomorphism of $\mathbb{R}^3$.
\end{itemize}

This leaves two cases to study: (A) unknotted toroidal sets with $\mathcal{N} \not\sim 1$, and (B) trivial toroidal sets with genus $0 < g(K) < +\infty$. We conclude this paper with two theorems that show that in both cases there exist toroidal sets that cannot be realized as attractors for a homeomorphism of $\mathbb{R}^3$. To do this we draw on the results of Section \ref{sec:construct}.
\medskip

A) \underline{\it Unknotted toroidal sets.} We already showed in Corollary \ref{cor:gen_sol_noatt} that, apart from $n$--adic solenoids, no generalized solenoid can be realized as an attractor. Notice that generalized solenoids are all unknotted toroidal sets because of the condition that the defining tori wind monotonically inside each other. The following theorem produces many more examples of unknotted toroidal sets that cannot be realized as attractors for homeomorphisms:

\begin{theorem} \label{teo:nonattracting1} Let $H$ be a feasible group. There exists an uncountable family $\{K_{\alpha}\}$ of toroidal sets such that:
\begin{itemize}
	\item[(i)] None of the $K_{\alpha}$ can be realized as an attractor for a homeomorphism of $\mathbb{R}^3$.
	\item[(ii)] The $K_{\alpha}$ are pairwise different (that is, not ambient homeomorphic).
	\item[(iii)] Each $K_{\alpha}$ is unknotted.
	\item[(iv)] Each $K_{\alpha}$ has $H$ as its first \v{C}ech cohomology group.
\end{itemize}
\end{theorem}
\begin{proof} We prove the theorem in the case when no element in $H$ is divisible by $2$. The other case (when every element in $H$ is divisible by $2$) is completely analogous. Let the family $\{P_{\alpha}\}$ run over all infinite sets $P_{\alpha}$ of prime numbers such that $2 \not\in P_{\alpha}$. Evidently $\{P_{\alpha}\}$ is an uncountable family. By Corollary \ref{cor:construct} for each $\alpha$ there exists an unknotted toroidal set $K_{\alpha}$ such that $\check{H}^1(K_{\alpha};\mathbb{Z}) = H$ and the prime divisors of $\mathcal{N}(K_{\alpha})$ are precisely the elements of $P_{\alpha}$. Since each $P_{\alpha}$ is infinite, it follows that $\mathcal{N}(K_{\alpha})$ is not number-like and then from Theorem \ref{teo:main1} we conclude that no $K_{\alpha}$ can be realized as an attractor for a homeomorphism of $\mathbb{R}^3$. Finally, since $P_{\alpha} \neq P_{\beta}$ for $\alpha \neq \beta$, by Proposition \ref{prop:local} the toroidal sets $K_{\alpha}$ and $K_{\beta}$ cannot be ambient homeomorphic (not even locally ambient homeomorphic).
\end{proof}

\medskip

B) \underline{\it Trivial toroidal sets with positive genus.} As mentioned in Section \ref{sec:construct} our understanding of this case is quite limited. Thus we will just consider toroidal sets with genus $g = 1$:

\begin{theorem} \label{teo:nonattracting2} There exists an uncountable family $\{K_{\alpha}\}$ of toroidal sets such that:
\begin{itemize}
	\item[(i)] None of the $K_{\alpha}$ can be realized as an attractor for a homeomorphism of $\mathbb{R}^3$.
	\item[(ii)] The $K_{\alpha}$ are pairwise different (that is, not ambient homeomorphic).
	\item[(iii)] Each $K_{\alpha}$ has genus $1$.
	\item[(iv)] Each $K_{\alpha}$ is trivial.
\end{itemize}
\end{theorem}
\begin{proof} The proof is completely analogous to that of Theorem \ref{teo:nonattracting1} but using Corollary \ref{teo:construct_trivial} to produce the $K_{\alpha}$.
\end{proof}

To conclude we would like to devote a few lines to yet another variation on the realizability problem. The result that a toroidal set that is an attractor for a homeomorphism has a finite genus depends crucially on the homeomorphism being defined on all of $\mathbb{R}^3$. Sometimes one is interested in dynamics that are defined only locally, and then one may want to consider a variation of the realizability problem where the (toroidal) set $K \subseteq \mathbb{R}^3$ is required to be an attractor only for a \emph{local} homeomorphism (rather than a homeomorphism defined on all of $\mathbb{R}^3$). For instance, consider an $n$--adic solenoid $K \subseteq \mathbb{R}^3$. This is clearly an attractor for a homeomorphism $f$ of $\mathbb{R}^3$. Now let $T$ be a solid torus that is a positively invariant neighbourhood of $K$ and let $e : T \longrightarrow e(T) \subseteq \mathbb{R}^3$ be an embedding of $T$ in a nontrivially knotted fashion. Denote by $K' := e(K)$. Evidently $K'$ is still an attractor for the local homeomorphism $efe^{-1} : e(T) \longrightarrow e(T)$. However, $K'$ has infinite genus (see Example \ref{ex:knotted}.(2)). Thus, it is indeed possible to have toroidal sets that can be realized as attractors for local homeomorphisms and have infinite genus (in particular, they cannot be realized as attractors for globally defined homeomorphisms).\footnote{Perhaps we should point out that this phenomenon does not occur for flows, because any local flow can be slowed down near the boundary of its domain of definition and extended (by letting it be stationary) to all of the ambient space.}

The self-index is more robust in this regard. The proof of Lemma \ref{lem:attapprox} works in any $3$--manifold and, as a consequence, it also works if we assume only that $f$ is a local homeomorphism (just replace $\mathbb{R}^3$ with the domain of $f$). In turn, the part of the proof of Theorem \ref{teo:main1} that shows that $\mathcal{N}(K)$ is number-like is also valid in such a situation. That is, the following generalization of Theorem \ref{teo:main1} holds:

\begin{theorem*} Let $K$ be a toroidal set that is an attractor for a local homeomorphism $f$ of a $3$--manifold. Then $\mathcal{N}(K)$ is number-like.
\end{theorem*}

Since the nonattracting nature of the toroidal sets $K_{\alpha}$ produced in Theorems \ref{teo:nonattracting1} and \ref{teo:nonattracting2} follows from the fact that $\mathcal{N}(K_{\alpha})$ has infinitely many prime divisors and hence is not number-like, in both cases the $K_{\alpha}$ cannot be realized even as attractors for local homeomorphisms.

\bibliographystyle{plain}
\bibliography{biblio}

\end{document}